\documentclass[adraft, copyright,creativecommons,sharealike,noncommercial]{eptcs}

\usepackage{import}
\usepackage{xspace} % Used to define all the relevant sets/categories in the macros

\usepackage{amsmath}
\usepackage{amsfonts}
\usepackage{amssymb} % Extram mathematical symbols
\usepackage{amsthm} % Theorem environments

\usepackage{mathtools} % Loads and extends amsmath

\usepackage{thmtools, thm-restate}
\usepackage{relsize} % Additional relative sizes for fonts
\usepackage{microtype} % Improves appearance of writing
\usepackage{graphicx} % To import graphics
\usepackage[usenames,dvipsnames]{xcolor} % Introduces colour names
\usepackage{tikz} % TikZ
\usepackage{tikzfig} % Additional TikZ macros for categorical quantum mechanics
\usepackage{circuitikz} % TikZ circuit diagrams 
\usetikzlibrary{
	arrows,
	shapes,
	decorations,
	intersections,
	backgrounds,
	positioning,
	circuits.ee.IEC
	}
\pgfplotsset{compat=1.10}
\usepackage{tikz-cd} % TikZ for commutative diagrams
	\declaretheorem[
	name=Theorem, 
	numberwithin=section %Change this if you want different numbering
	]{theorem}

	\declaretheorem[
	name=Proposition, 
	sibling=theorem % This make numbering to go along with the one above 
	]{proposition}

	\declaretheorem[
	name=Lemma, 
	sibling=theorem
	]{lemma}

	\declaretheorem[style=definition, 
	name=Remark, 
	sibling=theorem
	]{remark}

	\declaretheorem[
	style=definition, 
	name=Definition,
	sibling=theorem
	]{definition}

	\declaretheorem[
	style=definition, % Examples are in the same style of definitions
	name=Example
	]{example}

	\newcommand{\iffdef}{\stackrel{def}{\iff}} % If and only if (by definition)
	\newcommand{\Xcolour}{Red}
	\newcommand{\Zcolour}{YellowGreen}
	\newcommand{\Xaltcolour}{Purple}
	
	\newcommand{\Zaltcolour}{Cyan}
	
	\newcommand{\Dcolour}{black}
	\newcommand{\Zbwcolour}{white}
	\tikzset{->-/.style={decoration={markings,mark=at position #1 with {\arrow{>}}},postaction={decorate}}}
	\tikzset{-<-/.style={decoration={markings,mark=at position #1 with {\arrow{<}}},postaction={decorate}}}
\newcommand{\setdefinition}[1]{\ensuremath{#1}\xspace}

	\newcommand{\naturals}{\setdefinition{\mathbb{N}}} % Set of natural numbers
	\newcommand{\integers}{\setdefinition{\mathbb{Z}}} % Set of integer numbers
	\newcommand{\rationals}{\setdefinition{\mathbb{Q}}} % Set of rational numbers
	\newcommand{\reals}{\setdefinition{\mathbb{R}}} % Set of real numbers
	\newcommand{\complexs}{\setdefinition{\mathbb{C}}} % Set of complex numbers
	\newcommand{\Powerset}[1]{\mathcal{P}(#1)} % Powerset of #1
	\newcommand{\suchthat}[2]{\left\{#1 \: \colon \: #2\right\}} % Set of elements #1 such that condition #2 holds 
	
	\newcommand{\starNaturals}{\setdefinition{\nonstd{\naturals}}} % Set of non-standard naturals
	\newcommand{\starComplexs}{\setdefinition{\nonstd{\complexs}}} % Set of non-standard complex numbers
	\newcommand{\starReals}{\setdefinition{\nonstd{\reals}}} % Set of non-standard realsnumbers
	
	\newcommand{\Infinitesimals}{\mathbb{I}} % Ring of infinitesimal hyperreals
	\newcommand{\Limited}{\mathbb{L}} % Ring of limited hyperreals

	\newcommand{\nonstd}[1]{{^\star\!#1}} % Non standard extension

	\newcommand{\near}{\sim} % Limited distance relation symbol
	\newcommand{\infnear}{\simeq} % Infinitesimal distance relation symbol
	\newcommand{\Halo}[1]{\operatorname{Hal}(#1)} % Halo (monad) of a hyperreal.
	\newcommand{\Galaxy}[1]{\operatorname{Gal}(#1)} % Galaxy of a hyperreal
	\newcommand{\Shadow}[1]{\operatorname{Sh}(#1)} % Shadow of a hyperreal
\newcommand{\spacedefinition}[1]{\ensuremath{#1}\xspace}

\newcommand{\spacename}[1]{\spacedefinition{\mathcal{#1}}}

	\newcommand{\SpaceH}{\spacename{H}}

	\newcommand{\ket}[1]{\vert #1 \rangle} % Ket labelled #1
	\newcommand{\bra}[1]{\langle #1 \vert} % Bra labelled #1
	\newcommand{\Dim}[1]{\operatorname{dim}#1} % Dimension of a space
	\newcommand{\LtwoSym}{\operatorname{L}^2} % Symbol for L2 spaces
	\newcommand{\Ltwo}[1]{\spacedefinition{\LtwoSym[#1]}} % L2 space over space #1
\tikzstyle{env}=[copoint,regular polygon rotate=0,minimum width=0.2cm, fill=black]

\tikzstyle{probs}=[shape=semicircle,fill=white,draw=black,shape border rotate=180,minimum width=1.2cm]

\tikzstyle{every picture}=[baseline=-0.25em,scale=0.5]
\tikzstyle{dotpic}=[] % for backwards-compatibility
\tikzstyle{diredges}=[every to/.style={diredge}]
\tikzstyle{math matrix}=[matrix of math nodes,left delimiter=(,right delimiter=),inner sep=2pt,column sep=1em,row sep=0.5em,nodes={inner sep=0pt},text height=1.5ex, text depth=0.25ex]

\tikzstyle{inline text}=[text height=1.5ex, text depth=0.25ex,yshift=0.5mm]
\tikzstyle{label}=[font=\footnotesize,text height=1.5ex, text depth=0.25ex,yshift=0.5mm]
\tikzstyle{left label}=[label,anchor=east,xshift=1.5mm]
\tikzstyle{right label}=[label,anchor=west,xshift=-1.5mm]

\tikzstyle{braceedge}=[decorate,decoration={brace,amplitude=2mm,raise=-1mm}]
\tikzstyle{small braceedge}=[decorate,decoration={brace,amplitude=1mm,raise=-1mm}]

\tikzstyle{doubled}=[line width=1.6pt] % set the line width for all doubled (quantum) maps/wires
\tikzstyle{boldedge}=[doubled,shorten <=-0.17mm,shorten >=-0.17mm]
\tikzstyle{boldedgegray}=[doubled,gray,shorten <=-0.17mm,shorten >=-0.17mm]

\tikzstyle{semidoubled}=[line width=1.4pt] % set the line width for all doubled (quantum) maps/wires
\tikzstyle{semiboldedgegray}=[semidoubled,gray,shorten <=-0.17mm,shorten >=-0.17mm]

\tikzstyle{boldedgedashed}=[very thick,dashed,shorten <=-0.17mm,shorten >=-0.17mm]
\tikzstyle{vboldedgedashed}=[doubled,dashed,shorten <=-0.17mm,shorten >=-0.17mm]
\tikzstyle{left hook arrow}=[left hook-latex]
\tikzstyle{right hook arrow}=[right hook-latex]
\tikzstyle{sembracket}=[line width=0.5pt,shorten <=-0.07mm,shorten >=-0.07mm]

\tikzstyle{causal edge}=[->,thick,gray]
\tikzstyle{causal nondir}=[thick,gray]
\tikzstyle{timeline}=[thick,gray, dashed]

\tikzstyle{cedge}=[<->,thick,gray!70!white]

\tikzstyle{empty diagram}=[draw=gray!40!white,dashed,shape=rectangle,minimum width=1cm,minimum height=1cm]
\tikzstyle{empty diagram small}=[draw=gray!50!white,dashed,shape=rectangle,minimum width=0.6cm,minimum height=0.5cm]

\tikzstyle{dot}=[inner sep=0mm,minimum width=3mm,minimum height=3mm,draw,shape=circle,text depth=-0.1mm]
\tikzstyle{ddot}=[inner sep=0mm, doubled, minimum width=3.5mm,minimum height=3.5mm,draw,shape=circle]

\tikzstyle{black dot}=[dot,fill=black]
\tikzstyle{white dot}=[dot,fill=white,,text depth=-0.2mm]
\tikzstyle{green dot}=[white dot] % for backwards-compatibility
\tikzstyle{gray dot}=[dot,fill=gray!40!white,,text depth=-0.2mm]
\tikzstyle{red dot}=[gray dot] % for backwards-compatibility

\tikzstyle{black ddot}=[ddot,fill=black]
\tikzstyle{white ddot}=[ddot,fill=white]
\tikzstyle{gray ddot}=[ddot,fill=gray!40!white]

\tikzstyle{gray edge}=[gray!40!white]

\tikzstyle{small dot}=[inner sep=0.5mm,minimum width=0pt,minimum height=0pt,draw,shape=circle]

\tikzstyle{small black dot}=[small dot,fill=black]
\tikzstyle{small white dot}=[small dot,fill=white]
\tikzstyle{small gray dot}=[small dot,fill=gray!40!white]

\tikzstyle{causal dot}=[inner sep=0.4mm,minimum width=0pt,minimum height=0pt,draw=white,shape=circle,fill=gray!40!white]

\tikzstyle{phase dimensions}=[minimum size=5mm,font=\footnotesize,rectangle,rounded corners=2.5mm,inner sep=0.2mm,outer sep=-2mm,text height=1ex, text depth=0.25ex, yshift=0.5mm]
\tikzstyle{dphase dimensions}=[phase dimensions]
\tikzstyle{phase dot}=[dot,phase dimensions]

\tikzstyle{white phase dot}=[dot,fill=white,phase dimensions]
\tikzstyle{white phase ddot}=[ddot,fill=white,dphase dimensions]

\tikzstyle{white rect ddot}=[draw=black,fill=white,doubled,minimum size=5mm,font=\footnotesize,rectangle,rounded corners=2.5mm,inner sep=0.2mm]
\tikzstyle{gray rect ddot}=[draw=black,fill=gray!40!white,doubled,minimum size=6mm,font=\footnotesize,rectangle,rounded corners=3mm]

\tikzstyle{gray phase dot}=[dot,fill=gray!40!white,phase dimensions]
\tikzstyle{gray phase ddot}=[ddot,fill=gray!40!white,dphase dimensions]
\tikzstyle{grey phase dot}=[gray phase dot]
\tikzstyle{grey phase ddot}=[gray phase ddot]

\tikzstyle{cnot}=[fill=white,shape=circle,inner sep=-1.4pt]
\tikzstyle{hadamard}=[square box,inner sep=0 pt,font=\footnotesize,minimum height=4mm,minimum width=4mm]
\tikzstyle{dhadamard}=[hadamard,doubled]
\tikzstyle{antipode}=[white dot,inner sep=0.3mm,font=\footnotesize]

\tikzstyle{scalar}=[diamond,draw,inner sep=0.5pt,font=\small]
\tikzstyle{dscalar}=[diamond,doubled, draw,inner sep=0.5pt,font=\small]

\tikzstyle{small box}=[rectangle,inline text,fill=white,draw,minimum height=5mm,yshift=-0.5mm,minimum width=5mm,font=\small]
\tikzstyle{small gray box}=[small box,fill=gray!30]
\tikzstyle{medium box}=[rectangle,inline text,fill=white,draw,minimum height=5mm,yshift=-0.5mm,minimum width=10mm,font=\small]
\tikzstyle{square box}=[small box] % for backwards-compatibility
\tikzstyle{medium gray box}=[small box,fill=gray!30]
\tikzstyle{semilarge box}=[rectangle,inline text,fill=white,draw,minimum height=5mm,yshift=-0.5mm,minimum width=12.5mm,font=\small]
\tikzstyle{large box}=[rectangle,inline text,fill=white,draw,minimum height=5mm,yshift=-0.5mm,minimum width=15mm,font=\small]
\tikzstyle{large gray box}=[small box,fill=gray!30]

\tikzstyle{gray square point}=[small box,fill=gray!50]

\tikzstyle{dphase box white}=[dbox]
\tikzstyle{dphase box gray}=[dbox,fill=gray!50!white]

\tikzstyle{point}=[regular polygon,regular polygon sides=3,draw,scale=0.75,inner sep=-0.5pt,minimum width=9mm,fill=white,regular polygon rotate=180]
\tikzstyle{copoint}=[regular polygon,regular polygon sides=3,draw,scale=0.75,inner sep=-0.5pt,minimum width=9mm,fill=white]
\tikzstyle{dpoint}=[point,doubled]
\tikzstyle{dcopoint}=[copoint,doubled]

\tikzstyle{wide copoint}=[fill=white,draw,shape=isosceles triangle,shape border rotate=90,isosceles triangle stretches=true,inner sep=0pt,minimum width=1.5cm,minimum height=6.12mm]
\tikzstyle{wide point}=[fill=white,draw,shape=isosceles triangle,shape border rotate=-90,isosceles triangle stretches=true,inner sep=0pt,minimum width=1.5cm,minimum height=6.12mm,yshift=-0.0mm]
\tikzstyle{wide point plus}=[fill=white,draw,shape=isosceles triangle,shape border rotate=-90,isosceles triangle stretches=true,inner sep=0pt,minimum width=1.74cm,minimum height=7mm,yshift=-0.0mm]

\tikzstyle{wide dpoint}=[fill=white,doubled,draw,shape=isosceles triangle,shape border rotate=-90,isosceles triangle stretches=true,inner sep=0pt,minimum width=1.5cm,minimum height=6.12mm,yshift=-0.0mm]

\tikzstyle{tinypoint}=[regular polygon,regular polygon sides=3,draw,scale=0.55,inner sep=-0.15pt,minimum width=6mm,fill=white,regular polygon rotate=180] 

\tikzstyle{white point}=[point]
\tikzstyle{white dpoint}=[dpoint]
\tikzstyle{green point}=[white point] % for backwards-compatibility
\tikzstyle{white copoint}=[copoint]
\tikzstyle{gray point}=[point,fill=gray!40!white]
\tikzstyle{gray dpoint}=[gray point,doubled]
\tikzstyle{red point}=[gray point] % for backwards-compatibility
\tikzstyle{gray copoint}=[copoint,fill=gray!40!white]
\tikzstyle{gray dcopoint}=[gray copoint,doubled]

\tikzstyle{black point}=[point,fill=black]
\tikzstyle{black copoint}=[copoint,fill=black]

\tikzstyle{tiny gray point}=[tinypoint,fill=gray!40!white]

\tikzstyle{diredge}=[->]
\tikzstyle{rdiredge}=[<-]
\tikzstyle{thickdiredge}=[->, very thick]
\tikzstyle{pointer edge}=[->,very thick,gray]
\tikzstyle{pointer edge part}=[very thick,gray]
\tikzstyle{dashed edge}=[dashed]
\tikzstyle{thick dashed edge}=[very thick,dashed]
\tikzstyle{thick gray dashed edge}=[thick dashed edge,gray!40]
\tikzstyle{thick map edge}=[very thick,|->]

\makeatletter
\newcommand{\boxshape}[3]{%
\pgfdeclareshape{#1}{
\inheritsavedanchors[from=rectangle] % this is nearly a rectangle
\inheritanchorborder[from=rectangle]
\inheritanchor[from=rectangle]{center}
\inheritanchor[from=rectangle]{north}
\inheritanchor[from=rectangle]{south}
\inheritanchor[from=rectangle]{west}
\inheritanchor[from=rectangle]{east}
\backgroundpath{% this is new
\southwest \pgf@xa=\pgf@x \pgf@ya=\pgf@y
\northeast \pgf@xb=\pgf@x \pgf@yb=\pgf@y

\@tempdima=#2
\@tempdimb=#3

\pgfpathmoveto{\pgfpoint{\pgf@xa - 5pt + \@tempdima}{\pgf@ya}}
\pgfpathlineto{\pgfpoint{\pgf@xa - 5pt - \@tempdima}{\pgf@yb}}
\pgfpathlineto{\pgfpoint{\pgf@xb + 5pt + \@tempdimb}{\pgf@yb}}
\pgfpathlineto{\pgfpoint{\pgf@xb + 5pt - \@tempdimb}{\pgf@ya}}
\pgfpathlineto{\pgfpoint{\pgf@xa - 5pt + \@tempdima}{\pgf@ya}}
\pgfpathclose
}
}}

\boxshape{NEbox}{0pt}{5pt}
\boxshape{SEbox}{0pt}{-5pt}
\boxshape{NWbox}{5pt}{0pt}
\boxshape{SWbox}{-5pt}{0pt}
\boxshape{EBox}{-3pt}{3pt}
\boxshape{WBox}{3pt}{-3pt}
\makeatother

\tikzstyle{cloud}=[shape=cloud,draw,minimum width=1.5cm,minimum height=1.5cm]

\tikzstyle{map}=[draw,shape=NEbox,inner sep=2pt,minimum height=6mm,fill=white]
\tikzstyle{dashedmap}=[draw,dashed,shape=NEbox,inner sep=2pt,minimum height=6mm,fill=white]
\tikzstyle{mapdag}=[draw,shape=SEbox,inner sep=2pt,minimum height=6mm,fill=white]
\tikzstyle{mapadj}=[draw,shape=SEbox,inner sep=2pt,minimum height=6mm,fill=white]
\tikzstyle{maptrans}=[draw,shape=SWbox,inner sep=2pt,minimum height=6mm,fill=white]
\tikzstyle{mapconj}=[draw,shape=NWbox,inner sep=2pt,minimum height=6mm,fill=white]

\tikzstyle{medium map}=[draw,shape=NEbox,inner sep=2pt,minimum height=6mm,fill=white,minimum width=7mm]
\tikzstyle{medium map dag}=[draw,shape=SEbox,inner sep=2pt,minimum height=6mm,fill=white,minimum width=7mm]
\tikzstyle{medium map adj}=[draw,shape=SEbox,inner sep=2pt,minimum height=6mm,fill=white,minimum width=7mm]
\tikzstyle{medium map trans}=[draw,shape=SWbox,inner sep=2pt,minimum height=6mm,fill=white,minimum width=7mm]
\tikzstyle{medium map conj}=[draw,shape=NWbox,inner sep=2pt,minimum height=6mm,fill=white,minimum width=7mm]
\tikzstyle{semilarge map}=[draw,shape=NEbox,inner sep=2pt,minimum height=6mm,fill=white,minimum width=9.5mm]
\tikzstyle{semilarge map trans}=[draw,shape=SWbox,inner sep=2pt,minimum height=6mm,fill=white,minimum width=9.5mm]
\tikzstyle{semilarge map adj}=[draw,shape=SEbox,inner sep=2pt,minimum height=6mm,fill=white,minimum width=9.5mm]
\tikzstyle{semilarge map dag}=[draw,shape=SEbox,inner sep=2pt,minimum height=6mm,fill=white,minimum width=9.5mm]
\tikzstyle{semilarge map conj}=[draw,shape=NWbox,inner sep=2pt,minimum height=6mm,fill=white,minimum width=9.5mm]
\tikzstyle{large map}=[draw,shape=NEbox,inner sep=2pt,minimum height=6mm,fill=white,minimum width=12mm]
\tikzstyle{very large map}=[draw,shape=NEbox,inner sep=2pt,minimum height=6mm,fill=white,minimum width=17mm]

\tikzstyle{medium dmap}=[draw,doubled,shape=NEbox,inner sep=2pt,minimum height=6mm,fill=white,minimum width=7mm]
\tikzstyle{medium dmap dag}=[draw,doubled,shape=SEbox,inner sep=2pt,minimum height=6mm,fill=white,minimum width=7mm]
\tikzstyle{medium dmap adj}=[draw,doubled,shape=SEbox,inner sep=2pt,minimum height=6mm,fill=white,minimum width=7mm]
\tikzstyle{medium dmap trans}=[draw,doubled,shape=SWbox,inner sep=2pt,minimum height=6mm,fill=white,minimum width=7mm]
\tikzstyle{medium dmap conj}=[draw,doubled,shape=NWbox,inner sep=2pt,minimum height=6mm,fill=white,minimum width=7mm]
\tikzstyle{semilarge dmap}=[draw,doubled,shape=NEbox,inner sep=2pt,minimum height=6mm,fill=white,minimum width=9.5mm]
\tikzstyle{semilarge dmap trans}=[draw,doubled,shape=SWbox,inner sep=2pt,minimum height=6mm,fill=white,minimum width=9.5mm]
\tikzstyle{semilarge dmap adj}=[draw,doubled,shape=SEbox,inner sep=2pt,minimum height=6mm,fill=white,minimum width=9.5mm]
\tikzstyle{semilarge dmap dag}=[draw,doubled,shape=SEbox,inner sep=2pt,minimum height=6mm,fill=white,minimum width=9.5mm]
\tikzstyle{semilarge dmap conj}=[draw,doubled,shape=NWbox,inner sep=2pt,minimum height=6mm,fill=white,minimum width=9.5mm]
\tikzstyle{large dmap}=[draw,doubled,shape=NEbox,inner sep=2pt,minimum height=6mm,fill=white,minimum width=12mm]
\tikzstyle{large dmap conj}=[draw,doubled,shape=NWbox,inner sep=2pt,minimum height=6mm,fill=white,minimum width=12mm]
\tikzstyle{large dmap trans}=[draw,doubled,shape=SWbox,inner sep=2pt,minimum height=6mm,fill=white,minimum width=12mm]
\tikzstyle{very large dmap}=[draw,doubled,shape=NEbox,inner sep=2pt,minimum height=6mm,fill=white,minimum width=19.5mm]

\tikzstyle{muxbox}=[draw,shape=rectangle,minimum height=3mm,minimum width=3mm,fill=white]
\tikzstyle{dmuxbox}=[muxbox,doubled]

\tikzstyle{box}=[draw,shape=rectangle,inner sep=2pt,minimum height=6mm,minimum width=6mm,fill=white]
\tikzstyle{dbox}=[draw,doubled,shape=rectangle,inner sep=2pt,minimum height=6mm,minimum width=6mm,fill=white]
\tikzstyle{dmap}=[draw,doubled,shape=NEbox,inner sep=2pt,minimum height=6mm,fill=white]
\tikzstyle{dmapdag}=[draw,doubled,shape=SEbox,inner sep=2pt,minimum height=6mm,fill=white]
\tikzstyle{dmapadj}=[draw,doubled,shape=SEbox,inner sep=2pt,minimum height=6mm,fill=white]
\tikzstyle{dmaptrans}=[draw,doubled,shape=SWbox,inner sep=2pt,minimum height=6mm,fill=white]
\tikzstyle{dmapconj}=[draw,doubled,shape=NWbox,inner sep=2pt,minimum height=6mm,fill=white]

\tikzstyle{ddmap}=[draw,doubled,dashed,shape=NEbox,inner sep=2pt,minimum height=6mm,fill=white]
\tikzstyle{ddmapdag}=[draw,doubled,dashed,shape=SEbox,inner sep=2pt,minimum height=6mm,fill=white]
\tikzstyle{ddmapadj}=[draw,doubled,dashed,shape=SEbox,inner sep=2pt,minimum height=6mm,fill=white]
\tikzstyle{ddmaptrans}=[draw,doubled,dashed,shape=SWbox,inner sep=2pt,minimum height=6mm,fill=white]
\tikzstyle{ddmapconj}=[draw,doubled,dashed,shape=NWbox,inner sep=2pt,minimum height=6mm,fill=white]

\boxshape{sNEbox}{0pt}{3pt}
\boxshape{sSEbox}{0pt}{-3pt}
\boxshape{sNWbox}{3pt}{0pt}
\boxshape{sSWbox}{-3pt}{0pt}
\tikzstyle{smap}=[draw,shape=sNEbox,fill=white]
\tikzstyle{smapdag}=[draw,shape=sSEbox,fill=white]
\tikzstyle{smapadj}=[draw,shape=sSEbox,fill=white]
\tikzstyle{smaptrans}=[draw,shape=sSWbox,fill=white]
\tikzstyle{smapconj}=[draw,shape=sNWbox,fill=white]

\tikzstyle{dsmap}=[draw,dashed,shape=sNEbox,fill=white]
\tikzstyle{dsmapdag}=[draw,dashed,shape=sSEbox,fill=white]
\tikzstyle{dsmaptrans}=[draw,dashed,shape=sSWbox,fill=white]
\tikzstyle{dsmapconj}=[draw,dashed,shape=sNWbox,fill=white]

\boxshape{mNEbox}{0pt}{10pt}
\boxshape{mSEbox}{0pt}{-10pt}
\boxshape{mNWbox}{10pt}{0pt}
\boxshape{mSWbox}{-10pt}{0pt}
\tikzstyle{mmap}=[draw,shape=mNEbox]
\tikzstyle{mmapdag}=[draw,shape=mSEbox]
\tikzstyle{mmaptrans}=[draw,shape=mSWbox]
\tikzstyle{mmapconj}=[draw,shape=mNWbox]

\tikzstyle{mmapgray}=[draw,fill=gray!40!white,shape=mNEbox]
\tikzstyle{smapgray}=[draw,fill=gray!40!white,shape=sNEbox]

\makeatletter
\pgfdeclareshape{cornerpoint}{
\inheritsavedanchors[from=rectangle] % this is nearly a rectangle
\inheritanchorborder[from=rectangle]
\inheritanchor[from=rectangle]{center}
\inheritanchor[from=rectangle]{north}
\inheritanchor[from=rectangle]{south}
\inheritanchor[from=rectangle]{west}
\inheritanchor[from=rectangle]{east}
\backgroundpath{% this is new
\southwest \pgf@xa=\pgf@x \pgf@ya=\pgf@y
\northeast \pgf@xb=\pgf@x \pgf@yb=\pgf@y

\pgfmathsetmacro{\pgf@shorten@left}{\pgfkeysvalueof{/tikz/shorten left}}
\pgfmathsetmacro{\pgf@shorten@right}{\pgfkeysvalueof{/tikz/shorten right}}

\pgfpathmoveto{\pgfpoint{0.5 * (\pgf@xa + \pgf@xb)}{\pgf@ya - 5pt}}
\pgfpathlineto{\pgfpoint{\pgf@xa - 8pt + \pgf@shorten@left}{\pgf@yb - 1.5 * \pgf@shorten@left}}
\pgfpathlineto{\pgfpoint{\pgf@xa - 8pt + \pgf@shorten@left}{\pgf@yb}}
\pgfpathlineto{\pgfpoint{\pgf@xb + 8pt - \pgf@shorten@right}{\pgf@yb}}
\pgfpathlineto{\pgfpoint{\pgf@xb + 8pt - \pgf@shorten@right}{\pgf@yb - 1.5 * \pgf@shorten@right}}
\pgfpathclose
}
}

\pgfdeclareshape{cornercopoint}{
\inheritsavedanchors[from=rectangle] % this is nearly a rectangle
\inheritanchorborder[from=rectangle]
\inheritanchor[from=rectangle]{center}
\inheritanchor[from=rectangle]{north}
\inheritanchor[from=rectangle]{south}
\inheritanchor[from=rectangle]{west}
\inheritanchor[from=rectangle]{east}
\backgroundpath{% this is new
\southwest \pgf@xa=\pgf@x \pgf@ya=\pgf@y
\northeast \pgf@xb=\pgf@x \pgf@yb=\pgf@y

\pgfmathsetmacro{\pgf@shorten@left}{\pgfkeysvalueof{/tikz/shorten left}}
\pgfmathsetmacro{\pgf@shorten@right}{\pgfkeysvalueof{/tikz/shorten right}}

\pgfpathmoveto{\pgfpoint{0.5 * (\pgf@xa + \pgf@xb)}{\pgf@yb + 5pt}}
\pgfpathlineto{\pgfpoint{\pgf@xa - 8pt + \pgf@shorten@left}{\pgf@ya + 1.5 * \pgf@shorten@left}}
\pgfpathlineto{\pgfpoint{\pgf@xa - 8pt + \pgf@shorten@left}{\pgf@ya}}
\pgfpathlineto{\pgfpoint{\pgf@xb + 8pt - \pgf@shorten@right}{\pgf@ya}}
\pgfpathlineto{\pgfpoint{\pgf@xb + 8pt - \pgf@shorten@right}{\pgf@ya + 1.5 * \pgf@shorten@right}}
\pgfpathclose
}
}

\makeatother

\pgfkeyssetvalue{/tikz/shorten left}{0pt}
\pgfkeyssetvalue{/tikz/shorten right}{0pt}

\tikzstyle{kpoint common}=[draw,fill=white,inner sep=1pt,minimum height=3mm]
\tikzstyle{kpoint}=[shape=cornerpoint,shorten left=5pt,kpoint common]
\tikzstyle{kpoint adjoint}=[shape=cornercopoint,shorten left=5pt,kpoint common]
\tikzstyle{kpoint conjugate}=[shape=cornerpoint,shorten right=5pt,kpoint common]
\tikzstyle{kpoint transpose}=[shape=cornercopoint,shorten right=5pt,kpoint common]
\tikzstyle{kpoint symm}=[shape=cornerpoint,shorten left=5pt,shorten right=5pt,kpoint common]

\tikzstyle{black kpoint}=[shape=cornerpoint,shorten left=5pt,kpoint common,fill=black]
\tikzstyle{black kpoint adjoint}=[shape=cornercopoint,shorten left=5pt,kpoint common,fill=black]

\tikzstyle{kpointdag}=[kpoint adjoint]
\tikzstyle{kpointadj}=[kpoint adjoint]
\tikzstyle{kpointconj}=[kpoint conjugate]
\tikzstyle{kpointtrans}=[kpoint transpose]

\tikzstyle{big kpoint}=[kpoint, minimum width=1.2 cm, minimum height=8mm, inner sep=4pt, text depth=3mm]

\tikzstyle{wide kpoint}=[kpoint, minimum width=1 cm, inner sep=2pt, text depth=-0.7 mm]
\tikzstyle{wide kpointdag}=[kpointdag, minimum width=1 cm, inner sep=2pt, text depth=0.7 mm]
\tikzstyle{wide kpointconj}=[kpointconj, minimum width=1 cm, inner sep=2pt, text depth=-0.7 mm]
\tikzstyle{wide kpointtrans}=[kpointtrans, minimum width=1 cm, inner sep=2pt, text depth=0.7 mm]

\tikzstyle{gray kpoint}=[kpoint,fill=gray!50!white]
\tikzstyle{gray kpointdag}=[kpointdag,fill=gray!50!white]
\tikzstyle{gray kpointadj}=[kpointadj,fill=gray!50!white]
\tikzstyle{gray kpointconj}=[kpointconj,fill=gray!50!white]
\tikzstyle{gray kpointtrans}=[kpointtrans,fill=gray!50!white]

\tikzstyle{gray dkpoint}=[kpoint,fill=gray!50!white,doubled]
\tikzstyle{gray dkpointdag}=[kpointdag,fill=gray!50!white,doubled]
\tikzstyle{gray dkpointadj}=[kpointadj,fill=gray!50!white,doubled]
\tikzstyle{gray dkpointconj}=[kpointconj,fill=gray!50!white,doubled]
\tikzstyle{gray dkpointtrans}=[kpointtrans,fill=gray!50!white,doubled]

\tikzstyle{white label}=[draw,fill=white,rectangle,inner sep=0.7 mm]
\tikzstyle{gray label}=[draw,fill=gray!50!white,rectangle,inner sep=0.7 mm]
\tikzstyle{black label}=[draw,fill=black,rectangle,inner sep=0.7 mm]

\tikzstyle{dkpoint}=[kpoint,doubled]
\tikzstyle{wide dkpoint}=[wide kpoint,doubled]
\tikzstyle{dkpointdag}=[kpoint adjoint,doubled]
\tikzstyle{dkcopoint}=[kpoint adjoint,doubled]
\tikzstyle{dkpointadj}=[kpoint adjoint,doubled]
\tikzstyle{dkpointconj}=[kpoint conjugate,doubled]
\tikzstyle{dkpointtrans}=[kpoint transpose,doubled]

\tikzstyle{kscalar}=[kpoint common, shape=EBox, inner xsep=-1pt, inner ysep=3pt,font=\small]
\tikzstyle{kscalarconj}=[kpoint common, shape=WBox, inner xsep=-1pt, inner ysep=3pt,font=\small]

 \tikzstyle{upground}=[circuit ee IEC,thick,ground,rotate=90,scale=2.5]
 \tikzstyle{downground}=[circuit ee IEC,thick,ground,rotate=-90,scale=2.5]
 \tikzstyle{bigground}=[regular polygon,regular polygon sides=3,draw=gray,scale=0.50,inner sep=-0.5pt,minimum width=10mm,fill=gray]
\tikzstyle{arrs}=[-latex,font=\small,auto]
\tikzstyle{arrow plain}=[arrs]
\tikzstyle{arrow dashed}=[dashed,arrs]
\tikzstyle{arrow bold}=[very thick,arrs]
\tikzstyle{arrow hide}=[draw=white!0,-]
\tikzstyle{arrow reverse}=[latex-]
\tikzstyle{cdnode}=[]

\title{The way of the Infinitesimal}
\author{
	Fabrizio Genovese
	\institute{Quantum Group \\ University of Oxford}
	\email{fabrizio.genovese@hertford.ox.ac.uk}
}
\def\titlerunning{TBD}
\def\authorrunning{F. Genovese}

\begin{document}

\bibliographystyle{eptcs}
\maketitle

\begin{abstract}
In this document we provide a brief introduction to Non-Standard Analysis, with Categorical Quantum Mechanics in mind as application. We are convinced this document will be helpful for everyone wanting to learn how to practically use Non-Standard Analysis without having to go too deep into the logical and model-theoretical complications of the subject. Our goal is to give the reader the necessary confidence to manipulate infinities, infinitesimals and the standard constructions, to employ Non-Standard Analysis as an everyday tool in scientific research.

This document is heavily based on ``Lectures on the Hyperreals'' by Robert Goldblatt, that has to be credited for having given one of the simplest and friendliest explanations of the subject so far.
\end{abstract}

\section{Introduction} 
When we think about limits, derivatives and other constructions commonly employed in calculus we often do it in intuitive terms. We manipulate ``$dx$'', ``$dy$'' and the like as they were algebraic objects, with ease. Nevertheless, when it comes to formalize these entities mathematically the definition usually employed is the one \emph{\`a la Cauchy}: derivatives and limits become analytic constructions relying on the mantra ``the more you go further, the more this quantity will get near to this other one''. As an instance, take the usual undergrad textbook definition of limit for a succession, originally given by Weierstrass:
\begin{equation*}
	\lim_{n \to \infty} s_n = l \iffdef \forall \epsilon > 0, \, \exists n_0: \, \forall n \geq n_0, \, |l - s_n| \leq \epsilon
\end{equation*}
We all know what this definition tries to capture: It does not matter how small $\epsilon$ is, for every $\epsilon$ you pick at some point the distance between $s_n$ and $l$ will consistently be less than $\epsilon$, and hence we can say that $s_n$ gets ``closer and closer'' to $l$. We would be tempted to say, then, that $s_n$ becomes \emph{infinitesimally close to $l$} but we can't, since we have no notion of what ``infinitesimal'' really means: This concept has to be relegated to our intuition. In practice, this means that one has to deal with a lot of ugly inequalities, moving out from the comfortable and reassuring world of pure algebra (I'm pretty sure this is the main reason why people that hate calculus hate calculus).

Definitions such the one above work nicely for a plethora of applications, and this is clearly proved by the fact that calculus is probably the most popular tool in the history of Science. Nevertheless, not having any way to explicitly talk about infinities and infinitesimals becomes a problem when, say, one has to work with \emph{compositional features} of a model. 

To better explain this, we consider an example from Categorical Quantum Mechanics (abbreviated CQM). In this framework there are some very important tools that are called $\emph{Frobenius algebras}$. It is not important at this point to know what they do or why they are useful, what is important is that, in a Hilbert space \SpaceH, to have Frobenius algebras the sum $\sum_1^n \bra{e_i}$ has to be part of \SpaceH, where $n$ is $\Dim{\SpaceH$} and $\ket{e_i}$ are vectors of an orthonormal basis. Obviously using the standard definition of sum and limits this object does not belong to \SpaceH unless its dimension is finite (since that sum is divergent), and because of this CQM has been essentially relegated to the study of finite-dimensional systems for a long time. Similarly, in physics one manipulates all the time objects such as Dirac deltas and plain waves on spaces like \Ltwo{\reals} or \Ltwo{\complexs}. Again, it doesn't matter to know what Dirac deltas and plain waves are: The problem here is that they are not elements of \Ltwo{\reals} or \Ltwo{\complexs}, and this poses serious limitations when one wants to capture the compositional behaviour of physical systems using category theory.

The solution to this problem comes from Non-Standard Analysis, and our plan is more or less this: We want to extend the real numbers (or the complex numbers, or $\Ltwo{\reals}$, whatever) including infinities and infinitesimals. In this way entities like the diverging sum above become well-defined algebraic objects with their own dignity, and we can formalize our constructions (as in the CQM case) on these extended versions of our spaces. This construction will be explained in detail in the next section. This is especially appealing from the categorical point of view since, being everything well-defined, we can study and formalize the compositional features of our systems exactly as we would do in the finite-dimensional case.

This is clearly not enough: We want to be able to control our extension. This means that we want clear and sharp answers to questions like ``If $T$ is a theorem in some environment, do I know when it is also a theorem in the environment extension, and vice-versa?'' Satisfaction of this requirement allows us to go back and forth from the extension to the original environment, and is ultimately what characterizes our extension construction as meaningful, making us able to regard it as ``my original environment plus algebraic entities that represent infinities and infinitesimals, entities that I can manipulate exactly as my intuition would suggest''. The piece of machinery that puts us in control, attaining exactly the desiderata above, is called \emph{transfer theorem}, and it will be explained in detail in section~\ref{sec:transfer theorem}. 

In the last section we will put all this machinery into good use. We will take a lot of textbook examples \`a la Cauchy and we will re-express them ``\`a la Robinson'', that is, using non-standard formalism. This will hopefully give the practicing scientist in need of an algebraic formalization of calculus enough confidence to wander in the beautiful world of Non-Standard Analysis as he pleases.

As already mentioned in the abstract, we owe a lot to Robinson, the father of Non-Standard Analysis itself, and to Goldblatt, that was able to explain the subject in a very simple and intuitive way. In particular, this document can be considered a recap of Goldblatt's book ``Lectures on the Hyperreals''~\cite{Goldblatt1998}, to which we redirect the reader that wants to know more.

The original book by Robinson~\cite{Robinson1996} is, perhaps, still the Bible of Non-Standard Analysis: Countless non-standard formalisations of interesting mathematical structures can be found there, from Hilbert Spaces to Differential Geometry, and we redirect the reader to that beautiful world of wonders in case more in-dept information is needed.
\section{The construction of $\starReals$}\label{sec:construction of starreals}
In logic there is a very famous theorem called \emph{compactness theorem}. What it says goes more or less like this: If $\Sigma$ is a set of statements such that every finite subset $\Sigma'$ of $\Sigma$ has a model (that is, some mathematical structure in which all the statements in $\Sigma'$ are true), then $\Sigma$ has a model itself. Now, suppose that $\Sigma$ is the set of all statements that are true for real numbers, plus the statements
\begin{equation*}
	0 < \epsilon,\quad \epsilon < 1,\quad \epsilon < \frac{1}{2}, \quad \epsilon < \frac{1}{3}, \quad, \dots
\end{equation*}
Using the compactness theorem we can infer that $\Sigma$ has a model, call it $\starReals$, that will be an ordered field in which the element $\epsilon$ is an infinitesimal, that is, something greater than zero but smaller than any real number. We also get that statements in $\starReals$ ``formulated in the right way'' will hold if and only if they hold in $\reals$. This last sentence has been intentionally expressed in a murky way, and one of the objectives of this document is to understand what it precisely means.

Albeit the compactness theorem is what makes non-standard analysis ultimately work, extending the real numbers using this approach has a great disadvantage: We get a model of the reals containing infinities and infinitesimals, but in a non-constructive way. What we will do, instead, is to \emph{constructively build} the extended real numbers $\starReals$, using mathematical tools called \emph{ultraproducts}. The experienced logician will be aware that, in practice, what we will do amounts to re-proving the compactness theorem in a more specialized setting that suits our need. But in order to proceed with this construction, we have to figure out first what real numbers really are.

\subsection{Structure of $\reals$}

There are many definitions of real numbers, all equivalent. We here state some of them, with which the reader may be familiar:
\begin{itemize}
	\item A real number is something that can be expressed as an infinite decimal expression. As an instance,
	\begin{equation*}
		\pi = 3.141592\dots
	\end{equation*}
	The set of all such numbers is called $\reals$.
	
	\item A real number is an element of a field $\reals$ that is:
	\begin{itemize}
		\item Ordered, meaning that there is an order relation $\leq$ compatible with the field operations of $\reals$;
		\item Complete, meaning that every non-empty set $S \subseteq \reals$ with an upper bound\footnote{An upper bound for $S \subseteq \reals$ is an $x \in \reals$ such that for all $s \in S$ it is $s \leq x$.} has a supremum\footnote{This means that the set of the upper bounds of $S$ has a minimum.}.
	\end{itemize}
	 This notion identifies $\reals$ uniquely since it can be shown that two ordered complete fields are always isomorphic.
	 
	 \item Given a sequence $p_1, p_2, \dots, p_n, \dots$, we say that it is a \emph{Cauchy sequence} if $\lim_{i,j \to \infty} | p_i - p_j | = 0$. This is nothing more than the usual definition of Cauchy sequence seen in undergrad calculus. Now consider the set of all the Cauchy sequences of rational numbers (call it $\overline{\rationals}$), that is, all the Cauchy sequences in which $p_i \in \rationals$ for every $i$. 
	 
	 Given two sequences $p_1, p_2, \dots, p_n, \dots$, $s_1, s_2, \dots, s_n, \dots$, we say that they are \emph{equivalent} if
	 \begin{equation*}
	 	\lim_{n \to \infty} |p_n - s_n | = 0
	 \end{equation*}
	 meaning that two sequences are equivalent if they approach the same limit. It can be shown that the definition above is an equivalence relation on $\overline{\rationals}$. The equivalence classes of this relation can be endowed with an ordered field structure that is complete, and this again characterizes $\reals$ since we already said that two complete ordered fields are isomorphic.
	 
	 Note that in this approach two equivalent Cauchy sequences define the same real number.
 \end{itemize}
As we already mentioned, it can be proved that all the definitions above are equivalent, but the reader can already appreciate some differences: The third one is constructive, while the second one is not. We will extend $\reals$ with infinities and infinitesimals following steps similar to the ones in the third construction, starting with sequences of real numbers and quotienting them in an opportune way. The next section is about finding what this opportune way is.

\subsection{Comparing Sequences}
In the last subsection we said that we begin our construction considering sequences of real numbers. If $r$ is a sequence, we will indicate with $r_n$ the $n$-th element of the sequence, that is of course just a real number.

Given two sequences $r, s$, we want to find some criterion to compare them. Intuitively, one way to go is to say that $r, s$ are equivalent if they agree on a \emph{large number} of elements, that is, if the set
\begin{equation*}
	E_{r,s} := \suchthat{n \in \naturals}{r_n = s_n}
\end{equation*}
is \emph{large}, whatever this means. Since we want this notion of equivalence to define an equivalence relation, we can already state some properties we want this notion of ``bigness'' to have:
\begin{itemize}
	\item $\naturals$ has to be large, this is because our equivalence relation has to be reflexive, and clearly $E_{r,r} = \naturals$;
	\item Equivalence relations are transitive. This means that if $E_{r,s}$ and $E_{s,k}$ are large, also $E_{r,k}$ has to be large. But $E_{r,k}$ contains $E_{r,s} \cap E_{s,k}$, and hence the property we are looking for is 
	\begin{center}
		If $A, B$ are large and $A \cap B \subseteq C$, then $C$ is large.
	\end{center}
	In particular this means that the set of large sets is closed for intersection and supersets, meaning that
	\begin{center}
		If $A,B$ are large and $A \subseteq C$, then $A \cap B$ and $C$ are large. 
	\end{center} 
	\item $\emptyset$ cannot be large: Every other set $K \subset \naturals$ has the property $\emptyset \subseteq K$, so if $\emptyset$ were to be large every subset of $\naturals$ would be large, making our equivalence relation trivial.
\end{itemize}
\begin{remark}
	The first notion of largeness that comes to mind is to say that a set $L$ is large if it is \emph{cofinite}, meaning that its complement $\naturals - L$ is a finite set. Unfortunately this notion does not work well: What we want ultimately do is to take our extension of the real numbers to be the sequences of reals quotiented by the equivalence relation we are defining. Moreover we want this extension to be \emph{totally ordered}, meaning that given two equivalence classes of sequences $r,s$ it is always $[r] < [s]$, $[l] = [s]$ or $[l] > [s]$. The natural way to say that $[r]$ is less than a sequence $[s]$ is to require that the set\footnote{Obviously we have to check that this definition is independent of the representatives of $[r], [s]$, but this is an easy exercise.}
	\begin{equation*}
		L_{r,s} : = \suchthat{n \in \naturals}{r_n < s_n}
	\end{equation*}
	is large. Now consider the sequences $r = (1,0,1,0, \dots), s=(0,1,0,1,\dots)$. Clearly $E_{r,s} = \emptyset$, so these two sequences are not identified by our relation, and we should be able to say $[r] < [s]$ or $[r] > [s]$. But both $L_{r,s}, L_{s,r}$ are not cofinite, and so we are not able to compare $[r],[s]$ if we take cofiniteness as a definition of ``large''. 
\end{remark}
	From this remark we deduce that we have to impose an extra requirement, namely that
	\begin{center}
		Given a set $A$, either $A$ or $\naturals - A$ is large.
	\end{center}
	We can immediately infer that if we require this $A$ and $\naturals - A$ cannot be both large, otherwise $A \cap \naturals - A = \emptyset$ would be large as well.
	
	The experienced reader will already have noted that our requirements to define what largeness is match the definition of \emph{ultrafilter}. This is indeed the tool we are going to use.
	
	\subsection{Ultrafilters}
	After the discussion in the last section we are able to formally state what we want.
	\begin{definition}
		Let $I$ be a nonempty set, and denote with $\Powerset{I}$ its powerset (the set of all subsets of $I$). A \emph{filter} on $I$ is a nonempty collection $\mathcal{F}$ of subsets of $I$ (that is, an $\mathcal{F} \subseteq \Powerset{I}$) such that
		\begin{itemize}
			\item If $A, B \in \mathcal{F}$, then $A \cap B \in \mathcal{F}$ ($\mathcal{F}$ is closed with respect to intersections);
			
			\item If $A \in \mathcal{F}$ and $A \subseteq B \subseteq I$, then $B \in \mathcal{F}$ ($\mathcal{F}$ is closed with respect to supersets).
		\end{itemize}
	Moreover, a filter $\mathcal{F}$ is called \emph{ultrafilter} if it satisfies the additional condition
	\begin{center}
		For any $A \subseteq I$, it is $A \in \mathcal{F}$ or $I-A \in \mathcal{F}$.
	\end{center}
	Note that a filter containing $\emptyset$ coincides with $\Powerset{I}$. Thus, we call $\mathcal{F}$ \emph{proper} if $\emptyset \not\in \mathcal{F}$.
	\end{definition}
\begin{example}
	The powerset $\Powerset{I}$ of any nonempty set $I$ is a filter (the trivial one, if you want) but not an ultrafilter.
\end{example}
\begin{example}
	The \emph{cofinite} (or \emph{Fr\'echet}) filter, briefly discussed in the previous subsection in the case of natural numbers, is defined as
	\begin{equation*}
	\mathcal{F}:=\suchthat{N \subseteq I}{I - N \text{ is finite}}
	\end{equation*}
	This is a filter but not in general an ultrafilter, as we have already seen in the previous subsection.
\end{example}
\begin{example}
	The set $\{I\}$ is an ultrafilter. Note that $I$ is an element of every filter because of the superset closure, hence $\{I\}$ is the smallest filter on $I$.
\end{example}
\begin{example}
	The \emph{principal ultrafilter generated by an element $i \in I$}, defined as
	\begin{equation*}
	\mathcal{F}:=\suchthat{S \subseteq I}{i \in S}
	\end{equation*}
	is an ultrafilter (as the name suggests, I'd say\dots)
\end{example}

A very important fact, the proof of which can be found on~\cite[p.18 (1), p. 21 Cor.2.6.2]{Goldblatt1998}, is that
\begin{theorem}\label{thm:nonprincipalexists}
	Every filter on a finite set $I$ is a principal ultrafilter. If $I$ is infinite, then it admits a NON principal ultrafilter on it.
\end{theorem}
We will use this theorem later, so keep it in mind.

\subsection{Construction of $\starReals$}
Now we developed enough stuff to really begin our journey. As we said, we start from sequences of real numbers.
Since a sequence $(r_1, r_2, \dots)$ of real numbers can be seen as a function $\naturals \to \reals$ and vice-versa, we can denote with $\reals^\naturals$ the set of all such sequences. We will usually denote with $r$ the sequence $(r_1, r_2, \dots)$ (mnemonic rule: unless otherwise specified, such as in ``$r \in \reals$'', if $r$ has a subscript it's a real, otherwise it's a sequence of reals!), and with $\mathbf{r}$ the constant sequence $(r,r,\dots)$ for $r \in \reals$.

The first thing we want to do is to endow $\reals^\naturals$ with a ring structure. Given sequences $r,s$ we can define
\begin{align*}
	r \oplus s :&= (r_1 + s_1, r_2 + s_2, \dots)\\
	r \odot s :&= (r_1\cdot s_1, r_2 \cdot s_2, \dots) 
\end{align*}
$(\reals^\naturals, \oplus, \otimes)$ is a commutative ring, with zero $\mathbf{0} = (0,0,\dots)$, unity $\mathbf{1} = (1,1, \dots)$ and addictive inverses $-r = (-r_1, -r_2, \dots)$ for every $r$. It is not a field since $(0,1,0,1,\dots) \odot (1,0,1,\dots) = \mathbf{0}$
and so there are non-zero sequences that do not possess multiplicative inverses.

Our aim is to turn this ring into an ordered field. As espected, we will reach our goal quotienting $(\reals^\naturals, \oplus, \otimes)$ by an ultrafilter on the natural numbers. We will show in a bit that we need this ultrafilter to be non principal to get infinities and infinitesimals from our extension.
\begin{lemma}
	Let $\mathcal{F}$ be a non principal ultrafilter on $\naturals$, the existence of which is guaranteed by theorem~\ref{thm:nonprincipalexists}. The relation $\equiv$ on $\reals^\naturals$ defined by
	\begin{equation*}
		r \equiv s \iffdef \suchthat{n \in \naturals}{r_n = s_n} \in \mathcal{F}
	\end{equation*}
	is an equivalence relation, that is moreover compatible with the operations $\oplus, \odot$, meaning that if $r \equiv r', s \equiv s'$, then $r \oplus s \equiv r' \oplus s'$ and $r \odot s \equiv r' \odot s'$. 
\end{lemma}
When $r \equiv s$ we say that $r$ and $s$ \emph{agree almost everywhere modulo $\mathcal{F}$}. Borrowing some logical notation the set $\suchthat{n \in \naturals}{r_n = s_n}$ can be denoted as $[[r = s]]$, and can be thought as a measure of the extent to which ``$r=s$'' is true. Quotienting $\reals^\naturals$ by $\mathcal{F}$ can be tought as collapsing this measure into a binary one: $r, s$ agree ``enough to be equal'' or they do not. The notation $[[r = s]]$ is utterly convenient to generalise this procedure to statements other than equality, as an instance setting
\begin{equation*}
	[[r \leq s]] := \suchthat{n \in \naturals}{r_n \leq s_n}
\end{equation*}
Clearly we can now say that $r \leq s$ iff $[[r \leq s]] \in \mathcal{F}$. Moreover, it is easy to prove that if $r \equiv r', s \equiv s'$ then $[[r \leq s]] \in \mathcal{F}$ if and only if $[[r' \leq s']] \in \mathcal{F}$, meaning that also $\leq$ is somehow well-behaved with respect to $\equiv$. Generalization to $[[r < s]], [[r \geq s]], [[r > s]]$ is obvious. 

Since $\oplus, \odot, <$ are well-behaved with respect to $\equiv$, we can extend these operations to the quotient $\reals^\naturals/ \mathcal{F}$. We denote with $[r]$ an element of $\reals^\naturals/ \mathcal{F}$, that, being an equivalence class modulo $\mathcal{F}$, is the set of all $r' \in \reals^\naturals$ such that $r \equiv r'$. We set
\begin{align*}
	[r] + [s] &:= [r \oplus s]\\
	[r] \cdot [s] &:= [r \odot s]\\
	[r] < [s] &\iffdef r \leq s	
\end{align*}
\begin{theorem}
	$(\reals^\naturals/ \mathcal{F}, +, \cdot, <)$ is an ordered field. 
\end{theorem} 
\begin{proof}
	Showing that $\reals^\naturals/ \mathcal{F}$ is a ring with zero $[\mathbf{0}]$ and unit $[\mathbf{1}]$ is an easy exercise. To define multiplicative inverses, take an element $[r] \neq [\mathbf{0}]$. This means that $r$ is non-zero almost everywhere modulo $\mathcal{F}$, and so that $[[r = \mathbf{0}]] \notin \mathcal{F}$. Being $\mathcal{F}$ an ultrafilter, this implies that its complement $[[r \neq \mathbf{0}]] \in \mathcal{F}$.
	Setting
	\[
	s_n := \begin{cases}
	\frac{1}{r_n} \text{ if } n \in [[r \neq \mathbf{0}]]\\
	0 \text{ otherwise}
	\end{cases}
	\]
	then defines a sequence $s$ with the property that $[[r \odot s = \mathbf{1}]] = [[r \neq \mathbf{0}]] \in \mathcal{F}$, and hence $[r] \cdot [s] = [r \odot s] = [\mathbf{1}]$, meaning that $[s]$ is a multiplicative inverse of $[r]$.
	We now have to check that the order relation $<$ on $\reals^\naturals/ \mathcal{F}$ is total, meaning that for $[r],[s]$ it is always $[r]>[s]$, $[r]=[s]$ or $[r]<[s]$.
	Noting that $[[r > s]], [[r=s]], [[r < s]]$ are disjoint sets and that their union is the whole $\naturals$, the conclusion follows from a general property about ultrafilters (easy to prove!):
	\begin{center}
		If $A_1, A_2, \dots$ are pairwise disjoint sets such that $\bigcup_i A_i \in \mathcal{F}$, then $A_i \in \mathcal{F}$ \emph{for exactly one $i$}.
	\end{center}
	There are still some trifles to prove (like that the set of $[\mathbf{r}]$ such that $[\mathbf{0}] \leq [\mathbf{r}]$ is closed under addition), that are left to the reader. From now on, the ordered field $\reals^\naturals/ \mathcal{F}$ will be denoted with $\starReals$ and called the set of \emph{Hyperreals} (equivalently, \emph{Non-standard reals}). 
\end{proof}	
Importantly, we can send every real number $r \in \reals$ to the equivalence class $[\mathbf{r}] \in  {\starReals}$, where $\mathbf{r}$ is the constant sequence $(r,r, \dots)$. This correspondence is a ordered field homomorphism, and so we can regard $\reals$ as a subfield of $\starReals$: This tells us that the latter is an extension of the former in the ``intuitive'' sense. The common notation to define an element of $\starReals$ that corresponds to some $r \in \reals$ is $\nonstd{r}$. That is, $\nonstd{r} = [\mathbf{r}]$.

\subsection{What about $\mathcal{F}$?}
We built the hyperreal numbers as a quotient by some non principal ultrafilter $\mathcal{F}$. But which one? What if we have multiple choices of $\mathcal{F}$? Is there a right one to pick?

The answer to this question is \emph{it doesn't matter}: If we assume the \emph{continuum hypothesis} it can be proved that all the quotients of $\reals^\naturals$ by a non principal ultrafilter are isomorphic as ordered fields, and this is why we left the choice of $\mathcal{F}$ completely undetermined.

The continuum hypothesis is an important conjecture about cardinality of sets, and it has been proved by Paul Cohen that this conjecture is indecidable in the usual Zermelo-Fraenkel-Choice formalization of set theory. In particular, assuming it is true or not does not produce any contradiction, so it is up to us to decide if we want it or not in our axiomatization of set theory. Clearly we do, and in this way the problem of picking the right ultrafilter for our construction vanishes completely.

\subsection{Infinities and infinitesimals}
Let $\epsilon$ be the sequence $(1, \frac{1}{2}, \frac{1}{3}, \dots)$.
Clearly $[[\mathbf{0} \leq \epsilon]] = \naturals$ and hence is in $\mathcal{F}$. On the other hand, for every \emph{positive} $r \in \reals$ it is $[[\mathbf{0} \leq \mathbf{r}]]$; moreover the set $[[\epsilon < \mathbf{r}]]$ is cofinite. We now use another general property of ultrafilters (again, easy to prove!):
\begin{center}
	If $\mathcal{F}$ is non principal, it contains all the cofinite sets.
\end{center}
From this we infer that $[[\epsilon < \mathbf{r}]] \in \mathcal{F}$ for every positive real $r$, and hence we conclude that $[\epsilon]$ is a \emph{positive infinitesimal}. We can generalize this to define \emph{negative infinitesimals} in the obvious way, and we will refer to \emph{infinitesimals} in general when we do not specify a sign. If we set $[\omega] = [\epsilon]^{-1} = [(1,2,3,\dots)]$ then for every positive $r \in \reals$ the set $[[\mathbf{r} < \omega]]$ is again cofinite, and hence we can regard $[\omega]$ as a positive unlimited quantity. The existence of such $[\epsilon], [\omega]$ shows that $\starReals$ is a \emph{proper extension} of $\reals$, that is, it contains elements that cannot be put in correspondence with any real number.
\begin{remark}
	Note how having such $[\epsilon], [\omega]$ crucially depends on the fact that $\mathcal{F}$ is non principal. Otherwise, we would have $\mathcal{F}:= \suchthat{S \subseteq \naturals}{n \in S}$ for a fixed natural number $n$. But then any sequence $s = (s_1, s_2, \dots, ...)$ would agree with $(s_n, \dots, s_n)$ on the $n$-th number and hence $[[s = \mathbf{s_n}]] \in \mathcal{F}$, meaning that $[s] = [\mathbf{s_n}] = \nonstd{s_n}$. When $\mathcal{F}$ is principal, then, every element of $\reals/\mathcal{F}$ comes from a real number and we do not obtain anything new from our construction.
\end{remark}
If $r$ is a sequence that converges to zero, using the strategy above we can easily check that $[|r|]$ is an infinitesimal in $\starReals$ (note that, however, it is not guaranteed that $[r] = [\epsilon]$: there are many different infinities and infinitesimals in $\starReals$!). Similarly, if $r \to \infty$ then $[r]$ is unlimited, so the notion of infinities and infinitesimals in our extension $\starReals$ captures the intuitive concepts of infinity and infinitesimal that we know from undergrad calculus. 

Moreover, infinities (that from now on we will prefer to call ``unlimited quantities'') and infinitesimals are elements of $\starReals$, and hence we can multiply them, sum them etc. Their arithmetic will be explored in detail later.

\subsection{Extending sets, functions, relations}\label{sec:extending sets functions relations}
\begin{remark}
	For all the definitions given below one has to check that they are well defined, in the sense that they do not depend on the choice of representatives in an equivalence class $[r] \in \starReals$. This is indeed the case and verifying it can be a useful exercise.
\end{remark}
If $A$ is a subset of $\reals$, we can consider the set $\nonstd{A}$ defined by
\begin{equation*}
	[r] \in \nonstd{A} \iffdef \suchthat{n \in \naturals}{r_n \in A} \in \mathcal{F}
\end{equation*}
That is, $[r] \in \nonstd{A}$ if $r \in A$ almost everywhere. We call the elements of $\nonstd{A} - A$ \emph{non-standard}. As an example, using the definition it is easy to see that $[\omega] \in \starNaturals$, but clearly $[\omega] \neq [\mathbf{n}]$ for any natural $n$. Hence $[\omega]$ is a non-standard natural number.

Similarly, we can extend functions. If $f:\reals \to \reals$ then $\nonstd{f}:\starReals \to \starReals$ is defined as
\begin{equation*}
	\nonstd{f}([(r_1, r_2, \dots)]) = [(f(r_1), f(r_2), \dots]
\end{equation*}
We can also extend functions $f:A \to \reals$. To do this, for each $[r] \in \nonstd{A}$ define 
\begin{equation*}
	s_n :=
	\begin{cases}
		f(r_n) \text{ if } r_n \in A\\
		0 \text{ otherwise}
	\end{cases}
\end{equation*}
Then we set $\nonstd{f([r])} = [s]$. We need to do this because $[r] \in \nonstd{A}$ means that $r \in A$ \emph{almost everywhere}, and hence there could be some $r_n$ of the sequence $r$ that is not in $A$, on which $f$ is not well defined. With the definition above we formally set $f$ to $0$ on these elements. Notice moreover that if $r$ is a real number in $A$, then $\nonstd{f}({\nonstd{r}}) = {\nonstd{(f(r))}}$: $f$ and $\nonstd{f}$ agree on the real numbers. For this reason it makes sense to drop the $\nonstd{} $ symbol to denote the extension, and this is what we will do for most functions. As an instance, we will use $x^2$ both to denote the usual function $\reals \to \reals$ and its extension $\starReals \to {\starReals}$.

Note in particular that a sequence is a function $s:\naturals \to \reals$, hence we can extend it to an \emph{hypersequence} $s:\starNaturals \to \starReals$: We can now look which values it assumes at infinite since this corresponds to just unlimited numbers in $\starNaturals$. This is exactly what we will use to calculate limits.

Finally, we can extend real $k$-ary relations: A real $k$-ary relation $P$ is nothing but a subset of $\reals^k$, and hence we can set
\begin{equation*}
	\nonstd{P}([r^1], [r^k]) \iffdef \suchthat{n \in \naturals}{P(r^1_n, \dots r^k_n)} \in \mathcal{F}
\end{equation*}
It is easy to see that the previous extensions are just particular cases of this one: A set $A \subseteq \reals$ is just an unary real relation and every function $f:A \to \reals$ can be identified with the relation
$\suchthat{(a,r)}{f(a) = r}$. You can check that extending sets and functions seeing them as relations gives us the same results we got performing the extension directly.
\section{The transfer theorem}\label{sec:transfer theorem}
The \emph{transfer theorem} (also called \emph{transfer principle}, according to your taste) is what tells us which properties get preserved from statements on standard sets, like $\naturals$, $\reals$ etc. to their extensions $\starNaturals$, $\starReals$, \dots. Up to now we built our stuff ``manually'', having to verify every time some things: Namely, we said that some property held for a $[r]$ if it held almost everywhere on its components, that is, if the set of the $r_n$ for which the property held in the ``standard'' universe was in the ultrafilter $\mathcal{F}$. Clearly this ain't an easy way to do mathematics since at every step lots of conditions have to be checked and verified. The transfer theorem can be seen as a huge shortcut to avoid all this meticulous work.

To fully understand how the transfer theorem works quite a bit of a background in logic is needed. In here, we will provide a practical recipe with examples and counterexamples that should be more than sufficient to use it with confidence.

First of all we have to specify what do we mean when we say ``statement on a set''. 
Given a set $S$, we consider the triplet $(S,\mathfrak{R}_S, \mathfrak{F}_S)$, where $\mathfrak{R}_S$ and $\mathfrak{F}_S$ are the set of all the possible $k$-ary relations (for all finite $k$) and functions on $S$, respectively. Note that since a $k$-ary relation on $S$ is just a subset of $S^k$,  $S^k$ is a relation itself. $\mathfrak{R}_S$ then includes also all the subsets of $S$ (seen as unary relations) and all the finite products of subsets of $S$.

This triplets are called \emph{relational structures}, and each of them comes with a \emph{language}: 
This language is the set of all the well-formed formulas in which we allow the use of variables (ex. the $x$ in $x \in \naturals$), constants, that are just elements of $S$ (ex. the $1$ in $1 \in \naturals$), functional symbols in $\mathfrak{F}_S$ (ex. the $f$ in $f(x) \in \naturals$), relational symbols in $\mathfrak{R}_S$ (ex. the $\in$ in $x \in \naturals$) and the usual logical connectives $\wedge$, $\vee$, $\neg$, $\to$, $\leftrightarrow$, parentheses and ``$,$''. We moreover allow universal and existential quantification of variables as long as the quantifiers range on a relation in $\mathfrak{R}_S$.
This means that we allow \emph{only} quantification over elements of subsets of finite products of $S$, hence statements like $\forall K \in \Powerset{S}$ are in general not allowed since $\Powerset{S}$ cannot be expressed as a subset of $S^k$ for some finite $k$ if $S$ is infinite.

A \emph{statement on a set $S$} is a formula in the language of the relational structure on $S$ in which every variable is bounded by a quantifier, that is, every variable can ``range'' over a defined set.
\begin{example}
	Here some examples:
	\begin{itemize}
		\item $\forall K \subset \naturals, \emptyset \in K$ is not a formula in the language of $\naturals$, since quantification over $K$ is not allowed ($K$ is a subset of $\naturals$ and hence a variable ranging in $\Powerset{\naturals}$);
		\item $x \in \naturals$ is a formula in the language of $\naturals$, but not a statement since $x$ is not bounded;
		\item $1 \in \naturals$ is a statement in the language of $\naturals$ since there are no variables around, and hence everything is trivially bounded;
		\item $\forall x \in \naturals, x+1 \in \naturals$ is a statement in the language of $\naturals$ since every variable is bound;
		\item $\forall n \in \naturals, \exists r \in \reals : n<r$ is a statement in the language of $\reals$: both variables are bounded, one on a subset of $\reals$ and one on $\reals$ itself; It is not a statement in the language of $\naturals$ since we cannot define $\reals$ by means of subsets of finite products of $\naturals$.
		\item $\forall c \in \complexs, c+1 \in \reals$ is a statement (clearly false, but still a statement) in the language of $\reals$: $\complexs$ can be seen as $\reals^2$, and hence our quantification is legit! It is clearly also a statement in the language of $\complexs$.
	\end{itemize}
\end{example}
\subsection{$\nonstd{}$-transformations}
We saw in section~\ref{sec:construction of starreals} how to build $\starReals$ from $\reals$ by means of the ultrapower construction. This can be done for every set $S$, and we can transform a statement on $S$ to a statement on $\nonstd{S}$ using the tools developed in section~\ref{sec:extending sets functions relations} as follows:
\begin{itemize}
	\item We replace every set $A$ in the statement with its extension $\nonstd{A}$;
	\item We replace every function $f$ with its extension $\nonstd{f}$;
	\item We replace every constant $s$ with its interpretation in the extension $\nonstd{s}$.
\end{itemize}
As an instance, the statement $\forall x \in \naturals, x+1 \in \naturals$ becomes $\forall x (\nonstd{\in}) \starNaturals, x(\nonstd{+})\nonstd{1} (\nonstd{\in} )\starNaturals$. For some special relation and function symbols though, such as ``$+$'', ``$=$'', ``$\in$'', ``$\leq$'' and the like, we avoid to write explicitly the $\nonstd{}$ symbol to avoid clutter. In the previous example then we will usually write $\forall x \in \starNaturals, x+\nonstd{1} \in \starNaturals$. Note that this is not ambiguous: The $\in$ symbol, having $\starNaturals$ on the right, must necessarily be the extended version of $\in$, otherwise the formula would not make sense. Similarly, the $+$ is acting on elements $x, \nonstd{1}$ that are both in $\starNaturals$ and hence must denote the extension of the function $+$.

Finally, the relation between a statement and its $\nonstd{}$-transformation is expressed by the transfer theorem:
\begin{theorem}[Transfer theorem]
	If a property $\phi$ is true in the language of $S$, then $\nonstd{\phi}$ is true in the language of $\nonstd{S}$.
\end{theorem}
Taking $S$ to be $\reals$ and focusing only on universal quantification, we get the obvious corollary
\begin{theorem}[Universal transfer]
	If a property $\phi$ holds for \emph{all} real numbers, then $\nonstd{\phi}$ holds for \emph{all} hyperreal numbers.
\end{theorem}
If you think about it, this corollary is obvious: A property that holds for all real numbers is a statement on the language of the reals, hence its extension must hold, by transfer theorem, on all the hyperreals, exactly as the corollary says. It is nevertheless very useful:
\begin{example}
	We can now prove that $\starReals$ is a ordered field without having to check a lot of bullshit, simply considering
	\begin{align*}
	(\forall x,y,z \in \reals)((x+y)+z=x+(y+z)) &\qquad 
	(\forall x,y,z \in \reals)((x \cdot y) \cdot z=x \cdot (y \cdot z))\\
	(\forall x \in \reals)(x + 0 = x)\ &\qquad
	(\forall x \in \reals)(x \cdot 1 = x)\\
	(\forall x, y \in \reals)(x+y = y + x) &\qquad
	(\forall x, y \in \reals)(x \cdot y = y \cdot x)\\
	(\forall x \in \reals)( \exists y \in \reals)(x + y = 0) &\qquad
	(\forall x \in \reals)(x \neq 0 \rightarrow (\exists y \in \reals)x \cdot y = 1)\\
	(\forall x,y,z \in \reals)( x \cdot &(y + z) = x \cdot y + x \cdot z )\\
	(\forall x,y \in \reals)((x \leq &y) \vee (x = y) \vee (x \geq y))
	\end{align*}
	Since these formulas are always true in $\reals$, their $\nonstd{}$-transforms are also true in $\starReals$, proving that $\starReals$ is an ordered field. Note that $\starReals$ is not complete: the completeness property cannot, in fact, be expressed using only quantification over variables (you also need to quantify over subsets and the like), and hence cannot be transferred. If you recall that all ordered fields are isomorphic, then this should not be surprising: If $\starReals$ were to be complete it would have been isomorphic to $\reals$, and we wouldn't have obtained anything new!
\end{example}

Another interesting corollary of the Transfer theorem is
\begin{theorem}[Existential transfer]
	If there is an hyperreal satisfying some property $\nonstd{\phi}$, then there is some real satisfying $\phi$.
\end{theorem}
\begin{example}
	The existential transfer theorem can be very useful: As an example, take Rolle's theorem from undergrad calculus:
	\begin{center}
		If a function $f$ is continuous and differentiable on $[a,b]$ and $f(a) = f(b)$, then $f'(x) = 0$ for some $x \in [a,b]$.
	\end{center}
	This can be expressed saying that if $f$ is ``such and such'', then it exists a real number $x$ such that ``blablabla''. The idea is that if we can express the ``such and such'' and the ``blablabla'' part of the statement (that is, the bits involving derivatives, continuity and the like) as statements on the reals (and we obviously can, as we will see later), then we can consider the following
	\begin{center}
		If a function $\nonstd{f}$ is $\nonstd{}$-continuous and $\nonstd{}$-differentiable on $\nonstd{([ a,b])}$ and $\nonstd{f}(\nonstd{a}) = \nonstd{f}(\nonstd{b})$\\
		then $\nonstd{f}'(x) = \nonstd{0}$ for some $x \in \nonstd{([ a,b])}$.
	\end{center}
	And the two statements will be completely entangled: if one is true so is the other. The difference is that in the second case we have much more choice for our $x$, since it can be an hyperreal. The idea is that even if we find that an $x$ satisfying the condition turns out to be in $\starReals - \reals$, the transfer theorem ensures that there will be some $x' \in \reals$ for which the first statement holds.
\end{example}
The core of non-standard analysis, then, is the following: Every time we have a statement on $\nonstd{S}$ that is the $\nonstd{}$-transform of some statement $\phi$ on $S$, we can ``drag'' the result back to $S$. Clearly the statements on $\nonstd{S}$ that are not in the form $\nonstd{\phi}$ for some $\phi$ on $S$ are \emph{NOT} transferable, and this is exactly where one has to pay attention. 
\begin{example}
	Consider a sequence $s:\naturals \to \reals$ and extend it to $\nonstd{s}:\starNaturals \to \starReals$.
	Suppose we can show that the sequence $\nonstd{s}$ never takes infinite values. We then want to prove that $s$ is a bounded sequence. Since $\nonstd{s}$ never takes infinite values the following statement must be true for some unlimited $\omega$:
	\begin{equation*}
		(\forall n \in \starNaturals)(|\nonstd{s}(n)| \leq \omega)
	\end{equation*}
	Since $\omega$ is unlimited, though, it must be $\omega \in \starReals - \reals$ and hence the statement above is \emph{not} in a transferrable form, because $\omega$ cannot be expressed as $\nonstd{r}$ for some real $r$. We then try to massage the statement in a transferrable form: In fact, since $\omega \in \starReals$, the formula above implies that the following is true:
	\begin{equation*}
		(\exists r \in \starReals)(\forall n \in \starNaturals)(|\nonstd{s}(n)| \leq r)
	\end{equation*}
	but now we can transfer this back to our ``standard'' world! In fact, from here we can apply the transfer theorem and deduce that the following must be true too:
	\begin{equation*}
	(\exists r \in \reals)(\forall n \in \naturals)(|s(n)| \leq r)
	\end{equation*}	
	That obviously means that $s$ is bounded on the reals.
\end{example}
\begin{remark}[A bit of jargon]
	Consider now the triplet $(\nonstd{S}, \mathfrak{R}_{\nonstd{S}}, \mathfrak{F}_{\nonstd{S}})$, that is, the relational structure on $\nonstd{S}$. Relations in $\mathfrak{R}_{\nonstd{S}}$ that are of the form $\nonstd{R}$ for some $R$ in $\mathfrak{R}_{S}$ are called \emph{internal relations}. Clearly not all relations are internal, otherwise we could always transfer everything back and forth. Similarly we get a definition of \emph{internal function}. What the transfer theorem says is that a statement can be transferred from $\nonstd{S}$ to $S$ if and only if it comprises only internal relations and functions.
\end{remark}
\begin{center}
	\textbf{- Mantra of non-standard analysis -}\\
	\textbf{Pay attention!} Transfer theorem from $S$ to $\nonstd{S}$ is always ok. When you go in the other direction make sure, before applying it, that 100\% of your stuff is in transferrable form, that is, it is made of \emph{statements} (and not just formulas) in which only internal relations and functions are present.
\end{center}
\subsection{Aritmethic of $\nonstd{\reals}$}
First of all, some useful definitions. Since we know from section~\ref{sec:construction of starreals} that every real $s$ uniquely corresponds to an hyperreal $\nonstd{s}$, we will just call an hyperreal ``a real number'' when we mean that it is in the form $\nonstd{s}$ for some $s \in \reals$. An hyperreal $b$ is called
\begin{itemize}
	\item \emph{limited} if it is $r < b < r'$ for some $r,r' \in \reals$;
	\item \emph{positive unlimited} if $r < b$ for all $r \in \reals$;
	 \emph{negative unlimited} if $b < r$ for all $r \in \reals$;
	 \emph{unlimited} if it is positive or negative unlimited;
	\item \emph{positive infinitesimal} if $0 < b < r$ for all positive $r \in \reals$; 
	 \emph{negative infinitesimal} if $r < b < 0$ for all negative $r \in \reals$;
	 \emph{infinitesimal} if it ispositive or negative infinitesimal, or $0$;
	\item \emph{appreciable} if it is limited but not infinitesimal: $r < |b| < r'$ for some $r, r' \in \reals^+$.
\end{itemize}
From this it is clear that all reals and infinitesimal are limited, and that the only real infinitesimal is $0$. All reals are appreciable. You can think of appreciable number as an hyperreal that ``is not too fucked up'' with respect to real numbers.
\begin{proposition}
	Let $\epsilon, \delta$ be infinitesimals, $b,c,$ appreciable and $H, K$ unlimited. 	We have that the following rules hold\footnote{This is all pretty intuitive, just think about $b,c$ as ``more or less normal'', $\epsilon, \delta$ as ``damn close to 0'' and $H, K$ as ``incredibly big or incredibly small''. In particular the fact that $H+K$ is indeterminate is easy to understand if you consider that $H,K$ can have opposite signs. Who ``wins'' in this case? God knows.}:
	\begin{itemize}
		\item $-\epsilon, -\delta$ are infinitesimal, $-b, -c$ are appreciable and $-H, -K$ are unlimited;
		\item $\epsilon + \delta$ is infinitesimal, $b + c$ is limited (but possibly infinitesimal), $b+\epsilon$ is appreciable and $H+b$, $H+\epsilon$ are unlimited.
		\item $\epsilon \cdot \delta$ is infinitesimal, $b \cdot c$ is appreciable and $H \cdot b, H \cdot K$ are unlimited;
		\item $\frac{1}{\epsilon}$ is unlimited if $\epsilon \neq 0$, $\frac{1}{b}$ is appreciable and $\frac{1}{H}$ is infinitesimal;
		\item $\frac{\epsilon}{H}, \frac{\epsilon}{b}, \frac{b}{H}$ are infinitesimal, $\frac{b}{c}$ is appreciable, $\frac{H}{\epsilon}, \frac{b}{\epsilon}, \frac{H}{b}$ are unlimited;
		\item if $\epsilon, b, H$ are all $>0$, then $\sqrt[n]{\epsilon}, \sqrt[n]{b}$ and $\sqrt[n]{H}$ are infinitesimal, appreciable and unlimited, respectively;
		\item $\frac{\epsilon}{\delta}, \frac{H}{K}, \epsilon \cdot H$ and $H+K$ are undeterminate.
	\end{itemize}
	Note that infinitesimals and limited hyperreals form two rings (the first obviously without unit), denoted $\Infinitesimals$  and $\Limited$, respectively. $\Infinitesimals$ is also an ideal of $\Limited$.
\end{proposition}
\subsubsection{Halos, Galaxies, Shadows}\label{subsec:halos}
Given two hyperreals $r, s$, three things can happen:
\begin{itemize}
	\item $r$ and $s$ are \emph{infinitesimally close}, that is, $r-s$ is infinitesimal. In this case we write $r \infnear s$. This defines an equivalence relation and for every $r$ we define the \emph{halo} of $r$ as
	\begin{equation*}
		\Halo{r} := \suchthat{r' \in \starReals}{r' \infnear r}
	\end{equation*}
	$\Halo{0}$ is just $\Infinitesimals$, and it can be proved that for all $r \in \starReals$,  $\Halo{r}:= \suchthat{r + \epsilon}{\epsilon \in \Halo{0}}$.
	
	Note that Robinson has a different name for halos: he calls them \emph{monads} and denotes them with $\mu(r)$. Since one of the main goals of this tutorial is to be able to employ non-standandard analysis in a categorical context it is obvious that the latter notation is complete bollocks for us. So we will keep using the word ``halo'' to avoid ambiguity.
	\item  $r$ and $s$ are a limited distance near, that is, $r - s$ is limited. In this case we write $r \near s$. This defines an equivalence relation and for every $r$ we define the \emph{galaxy} of $r$ as
	\begin{equation*}
	\Galaxy{r} \iffdef \suchthat{r' \in \starReals}{r' \near r}
	\end{equation*}
	$\Galaxy{0}$ is just $\Limited$, and it can be proved that for all $r \in \starReals$,  $\Galaxy{r}:= \suchthat{r + l}{l \in \Galaxy{0}}$.
	\item The distance between $r$ and $s$ is unlimited. We are in general not interested in this situation.
\end{itemize}
One of the most important things we will use is the following:
\begin{theorem}
	For every limited hyperreal $r$, there is \emph{exactly one} real number in $\Halo{r}$, denoted $\Shadow{r}$ and called \emph{shadow of $r$}. That is, every limited hyperreal is infinitesimally close to exactly one real number. An alternative name for $\Shadow{r}$ is \emph{standard part of $r$}.
\end{theorem}
\begin{proof}
	Consider the set $A:=\suchthat{s \in \reals}{s < r}$. Since $r$ is limited there are reals $s',s''$ such that $s'' < r < s'$. The first part of the inequation tells us that $A$ is not empty, while the second tells us that it is bounded. Since $\reals$ is complete, $A$ has then a supremum, call it $a$. We want to prove that $a$ is infinitesimally close to $r$. 
	To do this we will prove that $|r - a| < e$ for all positive reals $e$ (definition of infinitesimal). 
	
	Since $a$ is a supremum for $A$ and for all $e \in \reals^+$ it is $a < a + e$, then $a+e$ can not be in $A$. This means $r \leq a+e$ by definition. 
	
	Now consider $a-e$. This is a real number and moreover $a-e \leq a$, so it cannot be $r \leq a-e$ otherwise $s < r \leq a-e$ against the fact that $a$ is the supremum of $A$. This implies
	$a-e < r$, and hence $a - e < r \leq a +e$, that means $|r-a| < e$, as we wanted.
	
	Now unicity. Suppose $s \infnear r$ for $s$ real. Recalling what we just proved we get $s \infnear r \infnear a$ and being $\infnear$ an equivalence relation, $s \infnear a$. This means $|s - a | < e$ for all positive reals $e$, but being $s, a$ reals also $| s - a |$ is real and the only infinitesimal real is $0$, proving $s =  a$.
\end{proof}
To mnemonically remember this last result you can think about a (not that good) song by Green Day, that in the chorus goes like ``My shadow's the only one that walks beside me''!
\begin{proposition}
	Shadows are very well behaved, and in fact the assignment $r \mapsto \Shadow{r}$ is a ordered field homomorphism from $\Limited$ to $\reals$. For every limited $r,s \in \starReals$ we indeed have:
	\begin{align*}
		\Shadow{r \pm s} = \Shadow{r} \pm \Shadow{s} & \qquad 
			\Shadow{r \cdot s} = \Shadow{r} \cdot \Shadow{s}\\
		(\forall n \in \naturals) (\Shadow{r^n} = \Shadow{r}^n)  & \qquad
			(\forall n \in \naturals, r \geq 0)(\Shadow{\sqrt[n]{r}} = \sqrt[n]{\Shadow{r}})\\	
			\Shadow{|r|} = |\Shadow{r}| & \qquad
				r \leq s \rightarrow \Shadow{r} \leq \Shadow{s}
	\end{align*}
	Noting moreover that the shadow of an hyperreal $r$ is $0$ if and only if $r$ is infinitesimal, we deduce that the kernel of the shadow map is just $\Infinitesimals$. The cosets of $\Infinitesimals$ in $\Limited$ are, for each $r \in \Limited$, the sets $\suchthat{r + \epsilon}{\epsilon \in \Infinitesimals = \Halo{0}}$, that is just $\Halo{r}$, so $\Halo{b} \mapsto \Shadow{b}$ is an injective monomorphism $\Limited/\Infinitesimals$ to $\reals$. Finally, recalling that every real number is the shadow od some limited hyperreal, we have that the correspondence $\Halo{b} \mapsto \Shadow{b}$ is also surjective, proving $\Limited/\Infinitesimals$ and $\reals$ to be isomorphic as ordered fields.
\end{proposition}
\subsubsection{Structure of the hypernaturals}
Hypernaturals (elements of $\starNaturals$) look like $\naturals$ only ``locally''. $\starNaturals$ can in fact be described as an abominous collage of densely ordered copies of $\integers$. To show this, first note that the only limited hypernaturals are the elements of $\naturals$. In fact if $n \in \starNaturals$ is limited we have $n \leq m$ for some $m \in \naturals$. Consider now the formula on $\naturals$
\begin{equation*}
	(\forall x \in \naturals)(x \leq m \rightarrow ((x=0) \vee (x=1) \vee \dots \vee (x = m)))
\end{equation*}
this formula is obviously true and then by transfer so it is its $\nonstd{}$-transform. Hence $n \leq m$ implies that $(n = 0)\vee \dots \vee (n = m)$, proving $n$ is an element of $\naturals$. From this we infer that $\starNaturals - \naturals$ consists only of positive unlimited numbers.

Given two hypernaturals $K, H$, if $K < H$ and $K \not \near H$ then $\Galaxy{K} < \Galaxy{H}$, meaning that every element in the set on LHS is less than every element in the RHS. 
If we define $\gamma(K) := \suchthat{H \in \starNaturals}{K \near H} = \Galaxy{K} \cap \starNaturals$, then again $K < H, K \not \near H$ implies $\gamma(K) < \gamma (H)$, while by definition $K \near H$ implies $\gamma(K) = \gamma(H)$. 

When $K$ is an unlimited hypernatural, $\gamma(K)$ can be also written as $\suchthat{K + m}{m \in \integers}$: this is mainly because if $K$ is positive unlimited so is $K + m$ for every integer $m$, and by definition $K + m \near K$. From here it is easy to see that, when $K$ is in $\starNaturals - \naturals$, $\gamma(K)$ isomorphic to the ordered \emph{set} \integers (they are not isomorphic as monoids though, because obviously $(K+m) + (K + n) = 2K + m + n$ and $2K = K + K \not\near K$ when $K$ is unlimited, hence $\gamma(K)$ is not a monoid).

When $K$ is limited, instead, $\gamma(K)$ is just $\naturals$. We then see that the order on $\starNaturals$ can be splitted into ``blocks'': There is an initial block order isomorphic to $\naturals$, followed by blocks $\gamma(K)$ each order isomorphic to $\integers$, for every $K \in \starNaturals - \naturals$.

If $K$ is unlimited, clearly $K < 2K$ and $K \not \near 2K$, hence $\gamma(K) < \gamma(2K)$ and so there is no ``bigger''block $\gamma(K)$. The blocks $\gamma(K)$ with $K$ unlimited also do not have a ``smaller'' block: obviously $\gamma(K) = \gamma(K+1)$, and by transfer theorem we can prove that, for each unlimited $K$, one between $K$ and $K+1$ is even, so we can assume without loss of generality to work with $\gamma(K)$ with $K$ even. Moreover, if $K$ is even, $\frac{K}{2} \in \starNaturals - \naturals$, $\frac{K}{2} < K$ and $\frac{K}{2} \not \near K$, meaning $\gamma\left(\frac{K}{2}\right) < \gamma(K)$. 

Finally we prove that these blocks $\gamma(K)$ with $K$ unlimited are \emph{densely ordered}, meaning that for every $K, H$ such that $\gamma(K) < \gamma(H)$ there is some $H'$ such that $\gamma(K) < \gamma(H') < \gamma(H)$. Again, we can just assume $K, H$ to be both even and infer $\gamma(K) < \gamma\left(\frac{H+K}{2}\right) < \gamma(H)$.

All things considered, the ordering on $\starNaturals$ looks more or less like this:
\begin{figure}[h]
	\centering
	\begin{tikzpicture}[scale=3]
	\node (Nt) at (0,-0.5) {$\gamma(1) = \naturals$};
	\node (K2t) at (1,0.5) {$\gamma(\frac{K}{2})$};
	\node (Kt) at (2,-0.5) {$\gamma(K)$};
	\node (K+Ht) at (3,0.5) {$\gamma\left(\frac{K+H}{2} \right)$};
	\node (Ht) at (4,-0.5) {$\gamma(H)$};
	\node (2Ht) at (5,0.5) {$\gamma(2H)$};
	
	%N 
	\draw[thick] (0, 0.3) -- (0,-0.3);
	% K2
	\draw (1, 0.23) -- (1,-0.23);
	\draw (0.9, 0.18) -- (0.9,-0.18);
	\draw (0.8, 0.15) -- (0.8,-0.15);
	\draw (0.7, 0.12) -- (0.7,-0.12);
	\draw (0.6, 0.09) -- (0.6,-0.09);
	\draw (0.5, 0.06) -- (0.5,-0.06);
	\draw (0.4, 0.03) -- (0.4,-0.03);
	
	\draw (1.1, 0.18) -- (1.1,-0.18);
	\draw (1.2, 0.15) -- (1.2,-0.15);
	\draw (1.3, 0.12) -- (1.3,-0.12);
	\draw (1.4, 0.06) -- (1.4,-0.06);	
	
	% K
	\draw[thick] (2, 0.3) -- (2,-0.3);
	\draw (1.9, 0.18) -- (1.9,-0.18);
	\draw (1.8, 0.15) -- (1.8,-0.15);
	\draw (1.7, 0.12) -- (1.7,-0.12);
	\draw (1.6, 0.06) -- (1.6,-0.06);
	
	\draw (2.1, 0.18) -- (2.1,-0.18);
	\draw (2.2, 0.15) -- (2.2,-0.15);
	\draw (2.3, 0.12) -- (2.3,-0.12);
	\draw (2.4, 0.06) -- (2.4,-0.06);
	
	% K+H
	\draw (3, 0.23) -- (3,-0.23);	
	
	\draw (2.9, 0.18) -- (2.9,-0.18);
	\draw (2.8, 0.15) -- (2.8,-0.15);
	\draw (2.7, 0.12) -- (2.7,-0.12);
	\draw (2.6, 0.06) -- (2.6,-0.06);
	
	\draw (3.1, 0.18) -- (3.1,-0.18);
	\draw (3.2, 0.15) -- (3.2,-0.15);
	\draw (3.3, 0.12) -- (3.3,-0.12);
	\draw (3.4, 0.06) -- (3.4,-0.06);
	
	% H
	\draw[thick] (4, 0.3) -- (4,-0.3);
	\draw (3.9, 0.18) -- (3.9,-0.18);
	\draw (3.8, 0.15) -- (3.8,-0.15);
	\draw (3.7, 0.12) -- (3.7,-0.12);
	\draw (3.6, 0.06) -- (3.6,-0.06);
	
	\draw (4.1, 0.18) -- (4.1,-0.18);
	\draw (4.2, 0.15) -- (4.2,-0.15);
	\draw (4.3, 0.12) -- (4.3,-0.12);
	\draw (4.4, 0.06) -- (4.4,-0.06);
	
	% 2H	
	\draw (5, 0.23) -- (5,-0.23);	
	\draw (4.9, 0.18) -- (4.9,-0.18);
	\draw (4.8, 0.15) -- (4.8,-0.15);
	\draw (4.7, 0.12) -- (4.7,-0.12);
	\draw (4.6, 0.06) -- (4.6,-0.06);
	
	\draw (5.1, 0.18) -- (5.1,-0.18);
	\draw (5.2, 0.15) -- (5.2,-0.15);
	\draw (5.3, 0.12) -- (5.3,-0.12);
	\draw (5.4, 0.09) -- (5.4,-0.09);
	\draw (5.5, 0.06) -- (5.5,-0.06);
	\draw (5.6, 0.03) -- (5.6,-0.03);
	\draw (4.1, 0.18) -- (4.1,-0.18);
	\draw (4.2, 0.12) -- (4.2,-0.12);
	\draw (4.3, 0.06) -- (4.3,-0.06);	

% Brace N
\draw [decorate,decoration={brace,amplitude=4pt}]
	(-1,-0.8) -- (1.2,-0.8) node [midway,right,xshift=.1cm] {}; 

\node (N1) at (-1,-1.6) {\tiny$0$};
\node (N2) at (-0.6,-1.6) {\tiny$1$};
\node (N3) at (-0.2,-1.6) {\tiny$2$};
\node (N4) at (0.2,-1.6) {\tiny$3$};
\node (N5) at (0.6,-1.6) {\tiny$4$};
\node (N6) at (1,-1.25) {$\dots$};

\draw[thick] (-1, -1) -- (-1,-1.5);
\draw (-0.6, -1) -- (-0.6,-1.5);
\draw (-0.2, -1) -- (-0.2,-1.5);
\draw (0.2, -1) -- (0.2,-1.5);
\draw (0.6, -1) -- (0.6,-1.5);

% Brace H	
\draw [decorate,decoration={brace,amplitude=4pt}, xshift=4cm]
(-1.4,-0.8) -- (1.4,-0.8) node [midway,right,xshift=.1cm] {};

\node (H3) at (2.8,-1.25) {$\dots$};
\node (H3) at (3.2,-1.6) {\tiny $H-2$};
\node (H4) at (3.6,-0.9) {\tiny$H-1$};
\node (H1) at (4,-1.6) {\tiny$H$};
\node (H2) at (4.4,-0.9) {\tiny$H+1$};
\node (H6) at (4.8,-1.6) {\tiny$H+2$};
\node (H6) at (5.2,-1.25) {$\dots$};

\draw (3.2, -1) -- (3.2,-1.5);
\draw (3.6, -1) -- (3.6,-1.5);
\draw[thick] (4, -1) -- (4,-1.5);
\draw (4.4, -1) -- (4.4,-1.5);
\draw (4.8, -1) -- (4.8,-1.5);
	\end{tikzpicture}
	\caption{Ordering of $\starNaturals$.}
\end{figure}

The order structure of $\starNaturals$ highlights one of the most important things to keep in mind in nonstandard analysis: making objects like infinities and infinitesimals algebraically manageable is not free, but comes at a price, that is, the topological structure of our spaces will in general become quite scrambled up. If the transfer theorem provides a sort of ``universal recipe'' for algebraic manipulation that can be learned with ease, there is no such analogous to address the topological issues, that have to be dealt with case by case. This mainly depends on the fact that topological properties are formalized in terms of formulas where the variables are open sets. But to apply transfer we are allowed, alas, to quantify only on elements and not on subsets, making many topological properties of $\starReals$, $\starComplexs$, $\starNaturals$ and the like not transferrable.

\section{Examples}
We have now developed enough material to start (re)formalizing analytical concepts. First, we will cover some standard textbook arguments in undergrad calculus, such as limits and derivatives. Then, we will focus on more high-level properties, like the topology of the extended reals and complexes, as well as non-standard Hilbert spaces.
\subsection{Calculus}
We start with something simple: Consider a sequence of real numbers $(r_1, r_2, \dots)$. We are used to say that this sequence \emph{converges to a limit $L$} if
\begin{equation}\label{eq:standardconvergence}
	\forall \varepsilon \in \reals^+ \exists m \in \naturals: \forall n \in \naturals, (n > m \to |r_n - L| < \epsilon)
\end{equation}
This is no more and no less than the definition of convergence for sequences of reals that everyone sees in a first undergrad calculus course. Conversely, in a nonstandard setting, our idea of convergence is that the sequence $r_n$ is \emph{infinitesimally close} to $L$ when $n$ is infinite. We finally have the tools to express this formally, obtaining
\begin{equation}\label{eq:nonstandardconvergence}
	\forall N \in \starNaturals - \naturals, r_N \in \Halo{L}
\end{equation}
This is great, especially if you thing that from a statement with three quantifier we now have a statement that just needs one. Obviously, we have to prove that the two definitions are actually the same:
\begin{proposition}\label{prop:sequenceequivalence}
	For a sequence of real numbers $(r_1, r_2, \dots)$ the standard and nonstandard definitions of convergence are equivalent, that is, \eqref{eq:standardconvergence} is true if and only if \eqref{eq:nonstandardconvergence} is true.
\end{proposition}
\begin{proof}
	$[std \Rightarrow nonstd]$ Fix an arbitrary $\varepsilon \in \reals^+$. Then, for some $m_\varepsilon \in \naturals$, equation \eqref{eq:standardconvergence} clearly implies 
	\begin{equation*}
		\forall n \in \naturals (n > m_\varepsilon \to |r_n - L| < \varepsilon)
	\end{equation*}
	Applying transfer then $\forall n \in \starNaturals (n > m_\varepsilon \to |r_n - L| < \varepsilon)$, and hence for every unlimited $N$ it is $|r_N - L| < \varepsilon$ since every unlimited hypernatural is by definition bigger than every natural number and hence also bigger than $m_\varepsilon$. But we already saw in the previous section that all the elements in $\starNaturals - \naturals$ are unlimited numbers, so we have 
	\begin{equation*}
		\forall N \in \starNaturals - \naturals (|r_N - L| < \varepsilon)
	\end{equation*}
	Our choice of $\varepsilon$ was arbitrary, so we conclude that for every unlimited hypernatural $N$ the inequation $|r_N - L| < \varepsilon$ holds for any $\varepsilon$, that means by definition $r_N \in \Halo{L}$. Hence $\forall N \in \starNaturals - \naturals, r_N \in \Halo{L}$, that is equation~\eqref{eq:nonstandardconvergence}. Notice how the key step here is not just applying transfer ``blindly'' to equation~\eqref{eq:standardconvergence}, but instead specialize it first such that we can interpret $\varepsilon$ and $m_\varepsilon$ as \emph{constants}. Applying transfer directly to equation~\eqref{eq:standardconvergence} would have given us $m_\varepsilon \in \starNaturals$, and hence we could not have guaranteed anymore that $N > m_\varepsilon$ for all $N \in \starNaturals - \naturals$.
	
	\medskip
	\noindent $[nonstd \Rightarrow std]$ Fix some $M \in \starNaturals - \naturals$. Since $n > M$ implies that $n$ is also unlimited and hence in $\starNaturals - \naturals$, equation~\eqref{eq:nonstandardconvergence} implies $\forall n \in \starNaturals, (n > M \to r_n \in \Halo{L})$. Now fix also an $\varepsilon \in \reals^+$. By definition, from $r_n \in \Halo{L}$ we have $| r_n - L | < \varepsilon$, and hence $\forall n \in \starNaturals, (n > M \to | r_n - L | < \varepsilon)$.
	
	This is still in a non-transferrable form because of the $M$ that gets in the way, but we can massage this formula observing that it obviously implies
	\begin{equation*}
		\exists m \in \starNaturals : \forall n \in \starNaturals, (n > m \to | r_n - L | < \varepsilon)
	\end{equation*}	
	But this is now the $\nonstd{}$-transform of a statement on $\reals$ and hence can be transferred. Finally, we get $\exists m \in \naturals : \forall n \in \naturals, (n > m \to | r_n - L | < \varepsilon)$, and since $\varepsilon$ was chosen arbitrarily in $\reals^+$ we can universally quantify $\varepsilon$ over $\reals^+$, getting equation~\eqref{eq:standardconvergence}.
\end{proof}
\subsection{Limits and continuity}
The definition of continuity in general settings is topological and requires universal quantification on open sets, making it difficult to transfer. Nevertheless, in metric spaces things are much easier, since the existence of a metric makes the usual topological definition of continuity equivalent to the following well-known first-order one (commonly known as \emph{$\delta$-$\epsilon$ definition})
\begin{equation}\label{eq:stdcontinuity}
	\forall \varepsilon \in \reals^+, \exists \delta \in \reals^+: \forall x \in \reals, (|x - c| < \delta \to |f(x) - f(c)| \leq \varepsilon)
\end{equation}
Intuitively, we say that a function $f$ is continuous at a point $c$ of its domain if ``the more $x$ gets near to $c$, the more $f(x)$ gets near to $f(c)$''. In the language of nonstandard analysis this can again be made precise, setting
\begin{equation}\label{eq:nonstdcontinuity}
	f \text{ continuous in $c$} \iffdef \forall x \in \starReals,(x \infnear c \to f(x) \infnear f(c)) \iff \forall x \in \starReals, f(\Halo{c}) \subseteq \Halo{f(c)}
\end{equation}
\begin{proposition}
	Equations~\eqref{eq:stdcontinuity} and~\eqref{eq:nonstdcontinuity} are equivalent.
\end{proposition}
\begin{proof}
	The proof is very similar in spirit to the one in~\ref{prop:sequenceequivalence}. To prove $[std \Rightarrow nonstd]$ just get rid of the first two quantifiers in~\eqref{eq:stdcontinuity} choosing any $\varepsilon$ and a $\delta_\epsilon$ depending on it and then transfer:
	\begin{equation*}
		\forall x \in \starReals, (|x - c| < \delta_\varepsilon \to |f(x) - f(c)| < \varepsilon)
	\end{equation*}
	Now, for all $x \in \starReals$ if $x \infnear c$ then clearly $|x - c| < \delta_\varepsilon$ and hence by the equation above $|f(x) - f(c)| \leq \varepsilon$. Since $\varepsilon$ was arbitrary then $x \infnear c$ implies $|f(x) - f(c)| < \varepsilon$ for any $\varepsilon$, hence $f(x) \infnear f(c)$.

	\medskip
	\noindent The direction $[nonstd \Rightarrow std]$ is as follows: Clearly if $\delta$ is a positive infinitesimal and $|x - c| < \delta$, then $|x - c| < \delta < r$ where $r$ is any positive real, and hence by definition $x \infnear c$. Then we get the inequation $\forall x \in \starReals, (|x - c| < \delta \to x \infnear c)$, where $\delta$ is a positive infinitesimal.
	Now pick any positive real $\varepsilon$. Clearly if $f(x) \infnear f(c)$ then $|f(x) - f(c)| < \varepsilon$.
	 Combining these last two facts with equation~\eqref{eq:nonstdcontinuity}, we get that if equation~\eqref{eq:nonstdcontinuity} is true, then also $\forall x \in \starReals, (|x - c| < \delta \to |f(x) - f(c)|<\varepsilon)$ is. $\delta$ is an infinitesimal and hence we still cannot transfer, but as usual this equation implies a weaker version in which $\delta$ gets replaced by a universal quantifier:
	\begin{equation*}
		\exists \delta \in \starReals: \forall x \in \starReals, (|x - c| < \delta \to |f(x) - f(c)|<\varepsilon)
	\end{equation*}
	And this is now transferrable. Remembering that $\varepsilon$ was an arbitrary positive real we can also add quantification of $\epsilon$ on $\reals^+$, finally obtaining equation~\ref{eq:stdcontinuity}.
\end{proof}
\begin{example}
	We now prove that the function $ x \mapsto x^2$ is continuous at a given point $c \in \reals$. Suppose $x \infnear c$, that is, $x = c + \epsilon$ for some infinitesimal $\epsilon$. Then, remembering the arithmetic rules for infinitesimals, 
	\begin{equation*}
		x^2 = (c + \epsilon)^2 = c^2 + 2c\epsilon + \epsilon^2 = c^2 + \text{ inifnitesimal } + \text{ infinitesimal} = c^2 + \text{ infinitesimal}
	\end{equation*}
	and hence $x \infnear c$ implies $x^2 \infnear c^2$. This clearly works way better than the usual definition!
\end{example}
Not surprisingly, using these techniques we can prove that the definitions given below are in fact equivalent to the usual definitions of limit for a function $f:A \to \reals$, where $A \subseteq \reals$.
\begin{definition}
	Let $f: A \to \reals$ with $A \subseteq \reals$ and let moreover be $c, L \in \reals$. We have (when present, a second line in the definition denotes the equivalent formal statement of RHS):
\begin{align*}
\lim_{x \to c} f(x) = L 
	&\iffdef \forall x \in \nonstd{A},((x \neq c \wedge x \infnear c) \to f(x) \infnear L)\\
\lim_{x \to +\infty} f(x) = L 
	&\iffdef \forall x \text{ positive unlimited in $\nonstd{A}$}, (f(x) \infnear L) \text{ (and such $x$ exists)}\\
	&\iffdef ((\nonstd{A} - A) \cap \starReals^+ \neq \emptyset) \wedge( \forall x \in (\nonstd{A} - A) \cap \starReals^+, (f(x) \infnear L))\\
\lim_{x \to c} f(x) = +\infty 
	&\iffdef \forall x \in \nonstd{A}, ((x \neq c \wedge x \infnear c) \to f(x) \text{ is positive unlimited})\\
	&\iffdef \forall x \in \nonstd{A}, ((x \neq c \wedge x \infnear c) \to f(x) \in (\nonstd{A} - A) \cap \starReals^+)
\end{align*}
\end{definition} 
Being the standard and nonstandard definitions of limit equivalent, it follows that the arithmetic of limits in the nonstandard case is absolutely identical to the standard one: limit of the sum is the sum of the limits, products of limits is limit of the products and so on.
\subsubsection{Derivatives}
The definition of derivative works exactly as it does in standard analysis: All things considered, after one defines the limit everything becomes easy.
\begin{definition}
	Given a function $f:A \to R$ and an infinitesimal $\epsilon$, we define the \emph{Newton quotient} as
	\begin{equation*}
		\mathcal{N}^f_\epsilon := \frac{f(x+\epsilon) - f(x)}{\epsilon}
	\end{equation*}
	We moreover say that \emph{$f$ has derivative $L$ at $x$} if $L$ is real and $\mathcal{N}^f_\epsilon \infnear L$ for all infinitesimals $\epsilon$.
\end{definition}
\noindent
Note that the Newton quotient, that is intuitively interpreted as the \emph{infinitesimal increment} of $f$ at $x$, depends also on $\epsilon$, since we have different infinitesimals in $\starReals$. Clearly if for some $x$ all these infinitesimal increments end up being limited and infinitesimally close one to each other, we can take $\Shadow{\mathcal{N}^f_\epsilon(x)}$ that is exactly the $L$ defined above. Using our definition of limit given previously we can moreover prove that this $L$ is indeed $\lim_{h \to 0}\frac{f(x+h) - f(x)}{h}$ and hence coincides ``usual'' derivative of $f$ at $x$ when $x$ is real, as we would expect.
\begin{example}
	This is probably the point when nonstandard analysis really shines: Derivatives are now just algebraic manipulations! Consider again the function $x \mapsto x^2$, and calculate
	\begin{equation*}
			\mathcal{N}^{x^2}_\epsilon := \frac{(x+\epsilon)^2 - x^2}{\epsilon} = \frac{x^2 + 2x\epsilon + \epsilon^2 - x^2}{\epsilon} = \frac{\epsilon(2x + \epsilon)}{\epsilon} = 2x + \epsilon
	\end{equation*}
Clearly for any $\epsilon$ it is $2x + \epsilon \in \Halo{2x}$ and hence $\Shadow{	\mathcal{N}^{x^2}_\epsilon (x)} = \Shadow{2x + \epsilon} = 2x$. This is indeed magic, we can finally calculate derivatives without having to resort to obnoxious limits!
\end{example}
Clearly all the usual textbook stuff is covered by non standard analysis (taylor series, derivatives in more variables, sequences of functions, integration, \dots), and we redirect the reader to~\cite{Goldblatt1998} if more information about these ``operative things'' is needed. Instead of pursuing this way, we will now turn our attention to other stuff, namely the nonstandard characterization of higher level structures such as nonstandard Hilbert spaces.
\section{Non-Standard separable Hilbert Spaces}

The construction of Non-standard separable Hilbert spaces goes exactly as for \reals: We can in fact apply the ultrafilter construction to any Hilbert space \SpaceH, obtaining an extension $\nonstd{\SpaceH}$. Now, do you remember all the fuss about relational structures carried out in section~\ref{sec:transfer theorem}? Well, the theory of Hilbert spaces can be formulated in terms of relational structures, and hence our transfer theorem keeps holding. We won't dive into details here (you can find a lot about non-standard Hilbert spaces in~\cite[Chapter 7, section 2]{Robinson1996}), but here is a useful fact you have to remember:

\begin{center}
	\item If $\SpaceH$ is a separable Hilbert space on \complexs (respectively, on \reals), then $\nonstd{\SpaceH}$ is an Hilbert space on \starComplexs (respectively \starReals).
\end{center}
In particular, the non-standard extension of the inner product on \SpaceH will be the inner product on $\nonstd{\SpaceH}$, so our inner product is now extended and goes from $\nonstd{\SpaceH} \times \nonstd{\SpaceH}$ to \starComplexs. Hence taking the inner product of a couple of vectors can now give an unlimited or an infinitesimal complex number as a result. Since we use the inner product to define a norm , this in  particular means that we will have vectors of unlimited or infinitesimal lenght. 
Moreover, a vector in a separable Hilbert space can be seen as a sequence of complex numbers $(v_0, v_1 \dots)$, and as usual a non-standard vector is a sequence of standard vectors (hence a sequence of sequences) quotiented by the ultrafilter relation. It is easy to see then how the vector identified by the sequence of vectors $(1,1, \dots), (\frac{1}{2}, \frac{1}{2}, \dots), (\frac{1}{3}, \frac{1}{3}, \dots)$ ends up being the vector $(\epsilon, \epsilon, \dots)$, with $\epsilon$ infinitesimal. This vector is clearly infinitesimal in any components. Similarly we can promptly became aware of the existence of unlimited vectors.

We give some useful definitions to conclude this section:
\begin{itemize}
	\item The \emph{standard vectors} are the ones of the form $\nonstd{v}$ for some $v \in \SpaceH$. These vectors \emph{DO NOT} form a Hilbert space on \starComplexs. To see this, consider the standard vector $v = (1,1,\dots)$. Since \starComplexs contains infinitesimal, I can multiply this vector by an infinitesimal $\epsilon$, obtaining $\epsilon v = (\epsilon, \epsilon, \dots)$, that clearly is not in the form $\nonstd{v}'$ for some $v' \in \SpaceH$. Nevertheless standard vectors \emph{do form} a Hilbert space on \complexs.
	
	\item The \emph{near-standard vectors} are the ones of the form $w \in \Halo{v}$, with $v$ of the form $\nonstd{v}$ for some $v \in \SpaceH$. This means that a near-standard vector is, as the name says, infinitesimally close to some standard vector. Note that, as in the real case, the halo of a non-standard vector can contain at most one standard vector. Again, these vectors are not a Hilbert space on \starComplexs. For instance, consider the vector $v = (1+\epsilon, 1+\epsilon, \dots)$, with $\epsilon$ infinitesimal. This is clearly in the halo of $(1,1,\dots)$ and hence is near-standard. As in the previous case I can pick the unlimited complex $\omega = \frac{1}{\epsilon}$ and do $\omega v = (\frac{1}{\epsilon} + 1, \frac{1}{\epsilon} + 1, \dots) = (\omega + 1, \omega + 1, \dots)$. being $\omega$ unlimited, every $\omega + 1$ is an unlimited complex number, and hence $\omega v$ is not in the halo of any standard vector. Again, near-standard vectors do instead for a Hilbert space on $\complexs$. Even better, if we quotient the set of near-standard vectors by the relation $\simeq$ defined in section~\ref{subsec:halos}, we get back our original vector space $\SpaceH$! This amounts to ``squeeze all together'' the near-standard vectors that are infinitesimally close to each other. 
\end{itemize}
\section{Conclusion}\label{sec:conclusion}

You know Kung Fu now! Use it and conquer the world!
%
%\section{Conlusion}
%
%``I know kung fu.''	
\bibliography{biblio.bib}
\end{document}